\documentclass{amsproc}

\usepackage{latexsym,amssymb}
\usepackage{amsfonts}

\usepackage{amsmath,amsthm,amssymb}

\newtheorem{theor}{Theorem}

\newtheorem{corol}[theor]{Corollary}

\newtheorem{thm}{Theorem}[section]
\newtheorem{lem}[thm]{Lemma}
\newtheorem{cor}[thm]{Corollary}

\newtheorem{prop}[thm]{Proposition}
\theoremstyle{definition}
\newtheorem{defn}[thm]{Definition}
\newtheorem{notation}[thm]{Notation}
\newtheorem{conv}[thm]{Convention}
\newtheorem{rem}[thm]{Remark}

\begin{document}

\title[Geometric entropy]{Geometric entropy of geodesic currents on free groups}

\author[Ilya Kapovich]{Ilya Kapovich}

\address{\tt Department of Mathematics, University of Illinois at
  Urbana-Champaign, 1409 West Green Street, Urbana, IL 61801, USA
  \newline http://www.math.uiuc.edu/\~{}kapovich/} \email{\tt
  kapovich@math.uiuc.edu}

\author[Tatiana Nagnibeda]{Tatiana Nagnibeda}

\address{\tt
Section de math\'ematiques,
Universit\'e de Gen\`eve,
2-4, rue du Li\`evre, c.p. 64,
1211 Gen\`eve, Switzerland
\newline http://www.unige.ch/~tatiana/}
\email{\tt tatiana.smirnova-nagnibeda@unige.ch}

\thanks{The first author was supported by the NSF grants DMS-0603921
  and DMS-0904200. Both authors acknowledge the support of the Swiss National
  Foundation for Scientific Research. The first author also acknowledges the support of the Max-Planck-Institut f\"ur Mathematik, Bonn and the activity "Dynamical Numbers" organized there by Sergiy Kolyada.}

\begin{abstract}
A geodesic current on a free group $F$ is an $F$-invariant measure on the set $\partial^2 F$ of pairs of distinct points of the boundary $\partial F$. The main aim of this paper is to
introduce and study the notion of \emph{geometric entropy} $h_T(\mu)$ for a
geodesic current $\mu$ on a free group $F$ with respect to a point $T$
in the Outer Space $cv(F)$, $T$ thus being an $\mathbb R$-tree
equipped with a minimal free and discrete isometric action of $F$. The
geometric entropy $h_T(\mu)$ measures the slowest exponential decay
rate of the values of $\mu$ on cylinder sets in $T$, with respect
to the $T$-length of the segment defining such a cylinder.

We obtain an explicit formula for $h_{T'}(\mu_T)$, where $T,T'\in
cv(F)$ are arbitrary points and where $\mu_T$ is the Patterson-Sullivan
current corresponding to $T$, in terms of the volume entropy of $T$ and
the extremal distortion of distances in $T$ with respect to
distances in $T'$.

We conclude that for $T\in CV(F)$ (where $CV(F)\subseteq cv(F)$ is the
projectivized Outer space consisting of all elements of $cv(F)$ with
co-volume $1$) and for a Patterson-Sullivan current $\mu_T$
corresponding to $T$, the function $CV(F)\to \mathbb R$ mapping $T'$ to
$h_{T'}(\mu_T)$, achieves a strict global maximum at $T'=T$.

We also show that for any $T\in cv(F)$ and any geodesic current $\mu$
on $F$, we have $h_T(\mu)\le h(T)$, where $h(T)$ is the volume entropy
of $T$, and the equality is realized when $\mu=\mu_T$.
For points $T\in cv(F)$ with simplicial metric (where all edges have length one), we relate the geometric entropy of a current
and the measure-theoretic entropy.
\end{abstract}

\date{\today}



\maketitle



\section{Introduction}\label{intro}

In \cite{CV} Culler and Vogtmann introduced a free group analogue of the Teichm\"uller space of a hyperbolic surface now
known as Culler-Vogtmann's \emph{Outer space}. The Outer space proved to be a fundamental object in the study of
the outer automorphism group of a free group and of individual outer automorphisms.

Let $F$ be a free group of finite rank $k\ge 2$. The \emph{nonprojectivized Outer space $cv(F)$}
consists of all minimal free and discrete isometric actions of $F$ on $\mathbb R$-trees. Two trees in $cv(F)$ are
considered equal if there exists an $F$-equivariant isometry between them. Note that for every $T\in cv(F)$ the action
of $F$ on $T$ is cocompact. There are several topologies on $cv(F)$ that are all known to coincide~\cite{Pau89}: the
equivariant Gromov-Hausdorff convergence topology,  the point-wise translation length function convergence topology,
and the weak $CW$-topology (see Section~\ref{section:cv} below for more details). There is a natural continuous left action of
$Out(F)$ on $cv(F)$ that corresponds to pre-composing an action of $F$ on $T$ with the inverse of an automorphism of $F$.
One often works with the projectivized version $CV(F)$ of $cv(F)$, called the \emph{Outer space}, which consists of all
$T\in cv(F)$ such that the quotient graph $T/F$ has volume $1$. The space $CV(F)$ is a closed $Out(F)$-invariant subset of $cv(F)$.

A geodesic current is, in the context of negative curvature, a measure-theoretic generalization of the notion of a free homotopy class of a closed curve on a
surface and of the notion of a conjugacy class in a group.
Let $\partial F$ be the hyperbolic boundary of $F$ and let $\partial^2 F$ be the set of all pairs
$(\xi,\zeta)\in \partial F\times\partial F$ such that $\xi\ne \zeta$. There is a natural left translation action of
$F$ on $\partial F$ and hence on $\partial^2 F$. A \emph{geodesic current} on $F$ is a positive, finite on compact subsets,
Borel measure on
$\partial^2 F$ that is $F$-invariant. (One sometimes also requires currents to be invariant with respect to the "flip"
map $\partial^2 F\to \partial^2 F$, $(\xi,\zeta)\mapsto (\zeta,\xi)$, but we do not impose this restriction in this paper).
The space $Curr(F)$ of all geodesic currents on $F$ is locally compact and comes equipped with a natural continuous action of
$Out(F)$ by linear transformations.

The study of geodesic currents in the context of hyperbolic
surfaces was initiated by Bonahon~\cite{Bo86,Bo88}. Bonahon extended the notion of a geometric intersection number between
two (free homotopy classes of) closed curves on a hyperbolic surface to a symmetric and mapping-class-group invariant notion
of an intersection number between two geodesic currents. He also showed that the Liouville embedding of the Teichm\"uller space
into the space of projectivized geodesic currents extends to a topological embedding of Thurston's compactification of the
Teichm\"uller space.
The study of geodesic currents also proved useful in the context of free groups
(see, for example,~\cite{Ma,Ka,Ka1,Ka2,KL1,KL2,KL3,CHL,Fra}). Thus in \cite{Ka,Ka1} Kapovich constructed a
canonical Bonahon-type $Out(F)$-invariant continuous "intersection form" $I:cv(F)\times Curr(F)\to\mathbb R$.
In a recent paper~\cite{KL2} Kapovich and Lustig extended this intersection form to the "boundary" of $cv(F)$ and
constructed its continuous $Out(F)$-invariant extension $I:\overline{cv}\times Curr(F)\to\mathbb R$. Here $\overline{cv}(F)$
is the closure of $cv(F)$ in the equivariant Gromov-Hausdorff (or the length function) topology. It is known that $\overline{cv}(F)$
consists precisely of all the minimal \emph{very small} isometric actions of $F$ on $\mathbb R$-trees. The projectivization of
$\overline{cv}(F)$ gives the \emph{Thurston compactification} $\overline{CV}(F)=CV(F)\cup \partial CV(F)$ of the Outer space $CV(F)$.
Motivated by Bonahon's result, in \cite{KN} Kapovich and Nagnibeda constructed the \emph{Patterson-Sullivan map}
$CV(F)\to\mathbb P Curr(F)$ and proved that this map is an $Out(F)$-equivariant continuous embedding (here
$\mathbb P Curr(F)$ is the space of \emph{projectivized geodesic currents} on $F$). Since $\mathbb P Curr(F)$ is compact,
the closure of the image of $CV(F)$ under this map gives a compactification of $CV(F)$. However, unlike in the case of hyperbolic surfaces, this
compactification is not the same as Thurston's compactification $\overline{CV}(F)$ of $CV(F)$.
Kapovich and Lustig~\cite{KL1} proved moreover that there does not exist a continuous $Out(F)$-equivariant map
$\partial CV(F)\to \mathbb PCurr(F)$.

Let $T\in cv(F)$. Note that $T$ is a proper Gromov-hyperbolic geodesic
metric space. Denote by $\partial T$ the hyperbolic boundary of $T$
and by $\partial^2 T$ the set of all pairs $(\xi,\zeta)\in \partial^2
T$ such that $\xi\ne \zeta$. Thus for any $(\xi,\zeta)\in \partial^2
T$ there exists a unique bi-infinite (non-parameterized) oriented geodesic line
$[\xi,\zeta]\subseteq T$ in $T$ from $\xi$ to $\zeta$. We think of
$[\xi,\zeta]\subseteq T$ as the image of an isometric embedding from
$\mathbb R$ to $T$, with the correct choice of an orientation on
$[\xi,\zeta]$.

Since $F$ acts
discretely, isometrically and co-compactly on $T$, the orbit map (for
any basepoint in $T$) defines a quasi-isometry $q_T: F\to T$ (where
$F$ is taken with any word metric) and hence
a canonical $F$-equivariant homeomorphism $\partial q_T: F\to \partial
T$. In turn, $\partial q_T$ defines an $F$-equivariant homeomorphism
$\partial^2 q_T: \partial^2F\to\partial^2 T$.

We will use the homeomorphisms $\partial q_T$ and $\partial^2 q_T$
to identify $\partial F$ with $\partial T$ and, similarly,
$\partial^2 F$ with $\partial^2 T$.
We will often suppress this explicit identification.

The \emph{volume entropy} $h=h(T)$ is defined as
\[
h(T)=\lim_{R\to\infty}\frac{\log \#\{g\in F: d_T(x_0,gx_0)\le R  \}}{R},
\]
where $x_0\in T$ is a basepoint. It is well known  (\cite{Coor}) that the limit
always exists and does not depend on the choice of a basepoint $x_0\in
T$. It also coincides with the \emph{critical
  exponent of the Poincar\'e series} :
\[
\Pi_{x_0}(s):=\sum_{g\in F} e^{-s\ d_T(x_0,gx_0)}\ ,
\]
namely, $\Pi_{x_0}(s)$ converges for all $s>h$ and diverges for all $s<h$.
Moreover, for every $x_0\in T$,  as $s\to h+$, any weak limit $\nu$ of the probability measures
  \[
  \frac{1}{\Pi_{x_0}(s)} \sum_{g\in F} e^{-s\ d(x_0,gx_0)} {\rm Dirac}(gx_0).
  \]
is a measure supported on $\partial T$. The measure-class of $\nu$ is uniquely determined and does not depend on the
choice of $x_0$ or on the choice of a weak limit.
Any such $\nu$ is called \emph{a Patterson-Sullivan measure} on $\partial T=\partial F$ corresponding to $T\in cv(F)$.

Furman proved in \cite{Fur}, in a wider context of hyperbolic groups, that there exists a unique, up to a scalar multiple, $F$-invariant and
flip-invariant nonzero locally finite measure $\mu_T$ on $\partial^2 T$
in the measure class of $\nu\times \nu$.
 Such a measure $\mu_T$ is called a
\emph{Patterson-Sullivan current} for $T\in cv(F)$.
Since $\mu_T$ is unique up to a scalar multiple, its projective class
$[\mu_T]$ is called \emph{the projective Patterson-Sullivan
  current} corresponding to $T\in cv(F)$.
Moreover, Furman's results imply that for $T\in cv(F)$ the projective Patterson-Sullivan current corresponding to $T$ depends only on the projective class $[T]$ of $T$, and thus allow to define the Patterson-Sullivan map $CV(F)\to\mathbb P Curr(F) ; [T]\mapsto\mu_T$.
We refer the reader to \cite{KN} for a more detailed discussion.

Let $T\in cv(F)$. Let $x,y\in T, x\ne y$, and $[x,y]$ denote the unique simplicial geodesic between $x$ and $y$ in $T$.
Denote
\begin{gather*}
  Cyl_{[x,y]}^T=Cyl_{[x,y]}:= \{(\zeta_1,\zeta_2)\in \partial^2 F: [x,y]\subseteq
  [\partial q_T(\zeta_1),\partial q_T(\zeta_2)] \\ \text{ and the orientations on $[x,y]$ and on
    $[\zeta_1,\zeta_2]$ agree}\}
\end{gather*}
 the \emph{two-sided cylinder set
  corresponding to $[x,y]$}.

For a fixed $T\in cv(F)$, any current $\mu\in Curr(F)$ is uniquely
determined by its values on the cylinder sets $Cyl_{[x,y]}\subseteq
\partial^2 F$, where $[x,y]$ varies over all nondegenerate geodesic
segments in $T$.  Note that since $\mu$ is $F$-invariant, the value
$\mu(Cyl_{[x,y]})$ depends only on $\mu$ and the path which is the
image of $[x,y]$ in the quotient graph $T/F$.  The "weights"
$\mu(Cyl_{[x,y]})$ tend to $0$ as $d_T(x,y)\to\infty$, and in many
interesting cases, as for example in that of Patterson-Sullivan
currents, this convergence is exponential. We introduce the notion of
\emph{geometric entropy} $h_T(\mu)$ of $\mu$ with respect to $T$ to
measure the slowest exponential rate of decay of the weights
$\mu(Cyl_{[x,y]})$ as $d_T(x,y)$ tends to infinity. More precisely
(see Definition~\ref{defn:ge} below):
\[
h_T(\mu):=\liminf_{d_T(x,y)\to\infty} \frac{-\log \mu (Cyl_{[x,y]})
}{d_T(x,y)}.
\]

We first establish some basic properties of geometric entropy in Section~\ref{sect:ge}. In particular
$h_T(\mu)=h_T(c\mu)$ for any $c>0$, $\mu\in Curr(F)$, so that $h_T(\mu)$ depends only on the projective class of $\mu$.
We note that for a fixed $\mu\in Curr(F)$ the function $E_\mu:cv(F)\to \mathbb R$, $T\mapsto h_T(\mu)$ is continuous
(Proposition~\ref{prop:cont}).
On the other hand, for any $T\in cv(F)$, the function $h_T: Curr(F)\to \mathbb R$, $\mu\mapsto h_T(\mu)$ is highly discontinuous.
Indeed, there  is a dense subset in $Curr(F)$ consisting of so-called "rational" currents (see Definition 5.1 in
\cite{Ka1}), whose geometric entropy is zero.
On the other hand
there are many currents with positive geometric entropy.

We obtain an explicit formula for the geometric entropy of a Patterson-Sullivan current $\mu_T$ of $T\in cv(F)$ with respect to an
arbitrary $T'\in cv(F)$. The geometric entropy of $\mu_T$ with respect to $T$ coincides with the volume entropy $h(T)$.
 We then solve two types of extremal problems regarding maximal values of the geometric entropy with either the tree or
the current arguments fixed.
Our main results are the following.

\begin{theor}\label{A} (Corollary \ref{cor:ps} and Theorem \ref{thm:comp1}).
Let $T\in cv(F)$ and let $\mu_T\in Curr(F)$ be a Patterson-Sullivan current corresponding to $T$.
Let $h(T)$ be the volume entropy of $T$.
Then
\begin{enumerate}
\item  $h_T(\mu_T)=h(T)$;
\item for any $T'\in cv(F)$
  \[
h_{T'}(\mu_T)= h(T)\inf_{g\in F\setminus\{1\}}
\frac{||g||_T}{||g||_{T'}}=\frac{h(T)}{\displaystyle \sup_{g\in F\setminus\{1\}} \frac{||g||_{T'}}{||g||_T}} .
\]
\end{enumerate}
Here, for $f\in F$ and $T\in cv(F)$, $||f||_T:=\inf_{x\in T} d_T(x,fx)$ is the \emph{translation length} of $f$.
\end{theor}

It is known~\cite{Wh,Ka,Ka1} that
 \[
\inf_{g\in F\setminus\{1\}} \frac{||g||_T}{||g||_{T'}}=\min_{g\in F\setminus\{1\}} \frac{||g||_T}{||g||_{T'}},\quad \sup_{f\in F\setminus\{1\}}
\frac{||f||_T}{||f||_{T'}}=\max_{f\in F\setminus\{1\}} \frac{||f||_T}{||f||_{T'}}
\]
and, moreover, one can algorithmically find $g,f\in F$ realizing the above equalities in a
finite subset of $F\setminus\{1\}$ depending only on $T$. It follows that the geometric entropy as function of $T'$
admits a continuous and strictly positive extension to $\overline{cv}(F)$.

The extremal distortions of the trees $T$ and $T'$ with respect to each other which appear in Theorem~\ref{A} are key ingredients in
the recent construction by
Francaviglia and Martino~\cite{FM} of asymmetric metrics on the
Outer space. Their construction is inspired by
Thurston's work on mutual extremal stretching factors (extremal Lipshitz constants) of two points in the
Teichm\"uller space~\cite{Thurston}.

We further use Theorem~\ref{A} to compute extremal values of $h_{T'}(\mu_T)$ as function of $T'\in CV(F)$ and show that this function
achieves its strict maximum at $T$.
\begin{corol}\label{B} (Corollaries \ref{cor:tw} and \ref{cor:inf}).
Let $T,T'\in CV(F)$ be such that $T\ne T'$. Let $\mu_T\in Curr(F)$ be a Patterson-Sullivan current corresponding to $T$ and
let $h(T)$ be the volume entropy of $T$. Then
\begin{enumerate}
\item for any $T'\in CV(F)$ such that $T'\ne T$
  \[
h_{T'}(\mu_T)<h_T(\mu_T)=h(T);
\]
\item we have
  \[
\inf_{T'\in CV(F)}h_{T'}(\mu_T)=0.
\]
\end{enumerate}
\end{corol}

We then proceed to study the geometric entropy as a function of $\mu\in Curr(F)$. Although it is highly discontinuous, we compute
its maximal value. Given a current $\mu\in Curr(F)$, we consider a family of measures $\{\mu_x\}_{x\in T}$ on $\partial F$
defined by their values on {\it one-sided cylinder subsets} of
$\partial F$:
\[
Cyl_{[x,y]}^x := \{\xi\in \partial F:  \text{ the geodesic ray }
  [x,\partial_T(\xi)] \text{ in $T$ begins with } [x,y]\}\subseteq \partial F,
\]
\[\mu_x (Cyl_{[x,y]}^x) := \mu (Cyl_{[x,y]}) .\]
If $\mu\in Curr(F)$, $\mu\ne 0$ then there is $x\in T$ such that
$\mu_x\ne 0$ (it is enough to take a segment $[x,y])$ such that $\mu
(Cyl_{[x,y]})\ne 0$). Note however, that the action of $F$ on the set
of vertices of $T$ is not necessarily transitive and it may happen
that $\mu\ne 0$ but for some vertex $x$ of $T$ we have $\mu_x=0$.

If $\mu_T$ is a Patterson-Sullivan current corresponding to $T$ then $\mu_x$ is a Patterson-Sullivan measure on $\partial F$
corresponding to $T$ (see \cite{KN}).

\begin{theor}\label{C} (Theorem~\ref{thm:vol} and Corollary~\ref{cor:sharp}.)
Let $T\in cv(F)$ and let $h=h(T)$ be the volume entropy of
$T$.
\begin{enumerate}
\item Let $\mu\in Curr(F)$, $\mu\ne 0$, and let $x\in T$ be such that $\mu_x\ne 0$.
Then
\[
h_T(\mu)\le {\mathbf {HD}} _{\partial T}(\mu_x)\le h_T(\mu_T)=h(T) ,
\]
where ${\mathbf {HD}} _{\partial T}(\mu_x)$ is the Hausdorff dimension of $\mu_x$ with respect to $\partial T$ with the metric $d_x$
(definitions are recalled in the beginning of Section~\ref{sect:hd}).

\item For $T\in cv(F)$ denote by $[T]$ the \emph{projective class} of $T$ that is, the set of all $cT\in cv(F)$ where $c>0$.
If $T'\in cv(F)$ is such that $[T']\ne [T]$, then
\[
h_T(\mu_{T'})< h(T).
\]
\end{enumerate}
\end{theor}

Part (1) of Theorem~\ref{C} implies that
\[
h(T)=h_T(\mu_T)=\max_{\mu\in Curr(F)-\{0\}} h_T(\mu).
\]
As was observed in \cite{KKS}, if $T_A\in cv(F)$ is the Cayley graph of $F$ with respect to a free basis $A$ and if $T\in cv(F)$ is
arbitrary, then
\[
{\bf HD}_{\partial T}(m_A)=\frac{\log(2k-1)}{\lambda_A(T)}
\]
where $k\ge 2$ is the rank of $F$, where $m_A$ is the "uniform" measure on $\partial F$ corresponding to $A$, and where $\lambda_A(T)$
is the "generic stretching factor" of $T$ with respect to $A$. That is, for an element $w_n\in F$ obtained by a simple
non-backtracking random walk on $T_A$ of length $n$, we have $||w_n||_T/||w_n||_A\to \lambda_A(T)$ as $n\to\infty$. Note that
in this case $h(T_A)=\log(2k-1)$ and $m_A=(\mu_{T_A})_x$ for $x$ being the vertex of the Cayley graph $T_A$ of $F$ corresponding
to $1\in F$. Also, obviously $\lambda_A(T)\le \sup_{g\in F\setminus\{1\}}\frac{||g||_T}{||g||_A}$.  Thus part (2) of Theorem~\ref{C} agrees
with these observations since it says that
\[
h_{T}(\mu_{T_A})=\frac{\log(2k-1)}{\sup_{g\in F\setminus\{1\}}\frac{||g||_T}{||g||_A}}\le {\bf HD}_{\partial T}(m_A)=
\frac{\log(2k-1)}{\lambda_A(T)}.
\]

These observations also suggest that if $T_0\in cv(F)$ is arbitrary (not necessarily corresponding to a free basis) and
if $T\in cv(F)$, one can define the "generic stretching factor" of $T$ with respect to $T_0$ as $\lambda_{T_0}(T):=
\frac{h(T_0)}{{\bf HD}_{\partial T}(\mu_0)}$ where $\mu_0$ is a Patterson-Sullivan measure on $\partial F$ corresponding to $T_0$.

Combining Theorem~\ref{A} and Theorem~\ref{C} we obtain:

\begin{corol}\label{D}
Let $T,T'\in cv(F)$.
\begin{enumerate}
\item
\[
\inf_{g\in F\setminus\{1\}} \frac{||g||_T}{||g||_{T'}}\le \frac{h(T')}{h(T)}\le \sup_{g\in F\setminus\{1\}} \frac{||g||_T}{||g||_{T'}}.
\]
\item Suppose that $[T]\ne [T']$. Then
\[
\inf_{g\in F\setminus\{1\}} \frac{||g||_T}{||g||_{T'}}< \frac{h(T')}{h(T)}< \sup_{g\in F\setminus\{1\}} \frac{||g||_T}{||g||_{T'}}.
\]
\end{enumerate}
\end{corol}

Here $[T]$ denotes the projective class of a tree $T\in cv(F)$.
Part (2) of Corollary~\ref{D} implies that if $[T]\ne [T']$ and $h(T)=h(T')$ then
\[
\inf_{g\in F\setminus\{1\}} \frac{||g||_T}{||g||_{T'}}< 1<\sup_{g\in F\setminus\{1\}} \frac{||g||_T}{||g||_{T'}}.
\]
Thus there exist $g,f\in F\setminus\{1\}$ such that $||g||_T<||g||_{T'}$ and $||f||_T>||f||_{T'}$.
This statement provides an analogue of a theorem of Tad White~\cite{Wh} who proved a similar result for $CV(F)$, that is for the
situation where points in $cv(F)$ are normalized by co-volume.
The above inequality is an analogue of White's result for the situation where we normalize points of $cv(F)$ by volume entropy.

In Theorem~\ref{thm:support} we also bound the geometric entropy $h_T(\mu)$ by the exponential growth rate of the
support of $\mu$ (appropriately defined) and observe that currents with support of subexponential growth always have geometric
entropy equal to zero.

 Suppose that $T\in cv(F)$ is a simplicial tree with all edges of length one, so that we can think of $T$ as
 $\widetilde\Gamma$ for the finite graph $\Gamma=T/F$ with the standard simplicial metric, and without degree-one vertices.
 There is a natural shift map $\sigma:\Omega(\Gamma)\to\Omega(\Gamma)$ on the space $\Omega(\Gamma)$ of all semi-infinite reduced edge-paths in $\Gamma$, corresponding to erasing the first edge of a path.
 The pair $(\Omega(\Gamma),\sigma)$ is an irreducible subshift of finite type. For every finite reduced edge-path $v$ in $\Gamma$
 there is a natural cylinder set
 $Cyl_v\subseteq \Omega(\Gamma)$ consisting of all semi-infinite paths $\gamma\in \Omega(\Gamma)$ that have $v$ as an
 initial segment. We can think of $\Omega(\Gamma)$ as an analogue of the unit tangent bundle for $\Gamma$. There is also a natural
 affine correspondence (see \cite{Ka1} for more details) between the space of geodesic currents $Curr(F)$ and the space
 $\mathcal M(\Gamma)$ of finite shift-invariant measures on the space $\Omega(\Gamma)$. For $\mu\in Curr(F)$ the corresponding shift-invariant measure $\widehat\mu\in \mathcal M(\Gamma)$ is defined by
 the condition
 $\widehat\mu(Cyl_v):=
 \mu(Cyl_{[x,y]})$ where $v$ is an arbitrary finite reduced edge path in $\Gamma$ and where $[x,y]$ is a lift of $v$ to $T$.
 In this setting, for any $\mu\in Curr(F)$, we relate the geometric entropy $h_T(\mu)$ and the measure-theoretic entropy of
 $\widehat\mu$ (normalized to be a probability measure). We refer the reader to \cite{Kitchens} for a more detailed discussion
 regarding measure-theoretic entropy (also known as metric entropy or Kolmogorov-Sinai entropy) of shift-invariant measures on
 irreducible subshifts of finite type. In particular, it is known that for such subshifts the measure-theoretic entropy of a
 shift-invariant probability measure never exceeds the topological entropy of the shift and that the equality is
 realized by a unique shift-invariant probability measure called the \emph{measure of maximal entropy}. For a shift-invariant
 probability measure $\nu$ on $\Omega(\Gamma)$ we denote its measure-theoretic entropy by $\hbar(\nu)$.

 \begin{theor}\label{E} (Theorem~\ref{thm:ks} and Corollary~\ref{cor:ks}.)
 Let $T\in cv(F)$ be a simplicial tree with simplicial metric and let $\Gamma=T/F$ be the quotient graph (thus all edges in
 $T$ and $\Gamma$ have length one, and $\Gamma$ is a finite connected graph without degree-one vertices). Let $\mu\in Curr(F)$,
 $\mu\ne 0$, and let $\widehat\mu\in \mathcal M(\Gamma)$ be the corresponding shift-invariant measure on
 $\Omega(\Gamma)$, normalized to be a probability measure. Similarly, let $\mu_T\in Curr(F)$ be a Patterson-Sullivan current
 for $T$ and let $\widehat\mu_T\in \mathcal M(\Gamma)$ be the corresponding shift-invariant measure on $\Omega(\Gamma)$,
 normalized to have total mass one. Then we have

 \begin{enumerate}
 \item \ \
 $h_T(\mu)\le \hbar(\widehat\mu)\le h_{topol}(\Omega(\Gamma),\sigma)=h(T)=\hbar(\widehat\mu_T) ;$

 \item $h_T(\mu)=h(T)$ if and only if there is $c>0$ such that $c\mu=\mu_T$, that is, $[\mu]=[\mu_T]$ in $\mathbb PCurr(F)$.
 \end{enumerate}
 \end{theor}
(In particular $\widehat\mu_T$ is the measure of maximal entropy for $(\Omega(\Gamma),\sigma)$).
Note that Theorem~\ref{E} implies part (2) of Corollary~\ref{C} for the case where $T\in cv(F)$ has simplicial metric
(all edges have length one).

We believe that a version of Theorem~\ref{E} should hold for an arbitrary $T\in cv(F)$ and not just for a tree with simplicial
metric, that is, for the case where $\Gamma=T/F$ is an arbitrary finite metric graph without degree-one vertices. Indeed, in this
general case there is still a natural correspondence between $Curr(F)$ and the space of all $\mathbb R_{\ge 0}$-invariant measures
on the space $\Omega_\mathbb R(\Gamma)$ of all locally isometric embeddings $\gamma:[0,\infty)\to \Gamma$, where $\gamma(0)$ is not
required to be a vertex.
(The space $\Omega_\mathbb R(\Gamma)$ comes equipped with a natural $\mathbb R_{\ge 0}$-action). However, in the case of a
non-simplicial metric on $\Gamma$ this $\mathbb R_{\ge 0}$ action does not nicely match the "combinatorial" shift $\sigma$
from Theorem~\ref{E}, which creates unpleasant technical problems with the argument. Proving an analogue of Theorem~\ref{E}
for an arbitrary $T\in cv(F)$ requires first developing basic background machinery and formalism for analyzing
$\mathbb R_{\ge 0}$-invariant measures on $\Omega_\mathbb R(\Gamma)$ and matching this information with the combinatorial
description of geodesic currents used in this paper. We postpone such analysis till a later date.

Let us note however, in view of the above discussion, that one can directly define an analogue of the notion of {\it geometric
entropy for any shift-invariant} (or even not necessarily invariant) {\it measure on an irreducible subshift of finite type}, and study this
notion in its own right. Thus let $A$ be a finite alphabet, let $A^\omega$ be the full shift (the set of all semi-infinite
words in $A$), let $\sigma:A^\omega\to A^\omega$ be the standard shift map and let $\Omega\subseteq A^\omega$ be a subshift of
finite type. For any finite word $v$ that occurs as an initial segment of some element of $\Omega$ we define $Cyl_v$ to
be the set of all $\gamma\in \Omega$ that begin with $v$. If $\mu$ is a finite positive Borel measure on $\Omega$, we can define
its \emph{geometric entropy} as:
\[
h_{geom}(\mu)=\liminf_{|v|\to\infty} \frac{-\log \mu (Cyl_v)
}{|v|},
\]
where $|v|$ is the ordinary combinatorial length of the word $v$.
Similar to the results above, one can show that $h_{geom}(\mu)\le h_{topol}(\Omega)$ and that if $\mu$ is a shift invariant probability measure then $h_{geom}(\mu)\le \hbar(\mu)\le h_{topol}(\Omega)$. Also, similarly to Theorem~\ref{C}, one can show that $h_{geom}(\mu)\le {\bf HD}_d(\mu)$ with respect to restriction $d$ to $\Omega$ of the standard metric from $A^\omega$. In this paper we concentrate on the setting of geodesic currents since it is more invariant and allows us to avoid choosing $\Omega$, (which would correspond to fixing a particular simplicial tree $T\in cv(F)$).

The authors thank Chris Connell and Seonhee Lim for useful
conversations. We also thank the referee for helpful comments.

\section{The Culler-Vogtmann Outer Space}\label{section:cv}

The Culler-Vogtmann outer space, introduced by Culler and Vogtmann in
a seminal paper~\cite{CV}, is a free group analogue of the
Teichm\"uller space of a closed surface of negative Euler
characteristic. We refer the reader to the original paper \cite{CV}
and to a survey paper \cite{Vog} for a detailed discussion of the
basic facts listed in this section and for the further references.

\begin{defn}[Non-projectivized Outer Space]
Let $F$ be a finitely generated free group of rank $k\ge 2$.

The \emph{non-projectivized outer space} $cv(F)$ consists of all
minimal free and discrete isometric actions of $F$ on $\mathbb
R$-trees. Two such trees are considered equal in $cv(F)$ if
there exists an $F$-equivariant isometry between them.
The space $cv(F)$ is endowed with the equivariant Gromov-Hausdorff
convergence topology.
\end{defn}

It turns out that every $T\in cv(F)$ is uniquely determined by its
\emph{translation length function} $\ell_T: F\to \mathbb R$, where for
every
$g\in F$
\[
\ell_T(g)=||g||_T=\min_{x\in T} d_T(x,gx)
\]
is the \emph{translation length} of $g$. Note that
$\ell_T(g)=\ell_T(hgh^{-1})$ for every $g,h\in F$. Thus $\ell_T$ can
be thought of as a function on the set of conjugacy classes in $F$.
The space $cv(F)$ comes equipped with a natural left $Out(F)$-action
by homeomorphisms. At the length function level, if $\phi\in Out(F)$,
$T\in cv(F)$ and $g\in F$ we have
\[
\ell_{\phi T}([g])=\ell_T(\phi^{-1} [g]).
\]
It is known that the equivariant Gromov-Hausdorff topology on $cv(F)$
coincides with the pointwise convergence topology at the level of
length functions. Thus for $T_n,T\in cv(F)$ we have $\lim_{n\to\infty}
T_n=T$ if and only if for every $g\in F$ we have $\lim_{n\to\infty} \ell_{T_n}(g)=\ell_T(g)$.

Points of $cv(F)$ have a more explicit combinatorial description as
``marked metric graph structures'' on $F$:

\begin{defn}[Metric graph structure]
Let $\Gamma$ be a finite connected graph without degree-one and
degree-two vertices. A \emph{metric graph structure $\mathcal L$} on $\Gamma$
is a function $\mathcal L: E\Gamma\to \mathbb R$ such that for every
$e\in E\Gamma$ we have
\[
\mathcal L(e)=\mathcal L(e^{-1})>0.
\]

More generally, define  a \emph{semi-metric graph structure $\mathcal L$} on $\Gamma$
to be a function $\mathcal L: E\Gamma\to \mathbb R$ such that for every
$e\in E\Gamma$ we have
\[
\mathcal L(e)=\mathcal L(e^{-1})\ge 0.
\]
A semi-metric graph structure $\mathcal L$ on $\Gamma$ is
\emph{nondegenerate} if there exists a subforest $Z$ in $\Gamma$ such
that $\mathcal L(e)>0$ for every $e\in E\Gamma-EZ$.

For a semi-metric graph structure $\mathcal L$ on $\Gamma$ define the
\emph{volume} of $\mathcal L$ as
\[
vol(\mathcal L)=\sum_{e\in E^+\Gamma} \mathcal L(e)
\]
where $E\Gamma=E^+\Gamma\sqcup E^-\Gamma$ is any orientation on $\Gamma$.

If $\mathcal L$ is a semi-metric graph structure on $\Gamma$ and $v=e_1,\dots, e_n$ is an edge-path in $\Gamma$, we denote
\[
\mathcal L(v)=\sum_{i=1}^n \mathcal L(e_i)
\]
and call $\mathcal L(v)$ the $\mathcal L$-\emph{length} of $v$.
\end{defn}

\begin{conv}
For a (finite or infinite) graph $\Delta$, denote by $V\Delta$ the set
of all vertices of $\Delta$, and denote by $E\Delta$ the set of all
oriented edges of $\Delta$. Combinatorially, we use Serre's convention
regarding graphs. Namely, $\Delta$ comes equipped with three functions: $o:E\Delta\to V\Delta$; $t:E\Delta\to V\Delta$; and
${}^{-1}:E\Delta\to E\Delta$, such that $e^{-1}\ne e$,
$(e^{-1})^{-1}=e$ for every $e\in E\Delta$,
$o(e)=t(e^{-1})$, and $t(e)=o(e^{-1})$ for every $e\in E\Delta$. The
edge $e^{-1}$ is called the \emph{inverse} of $e$. An \emph{orientation} on $\Delta$ is a partition $E\Delta=E^+\Delta\sqcup E^-\Delta$, where for every $e\in E\Delta$ one
of the edges $e, e^{-1}$ belongs to $E^+\Delta$ and the other edge belongs to  $E^-\Delta$.

An \emph{edge-path} $\gamma$
in $\Delta$ is a sequence of oriented edges which connects a vertex
$o(\gamma)$ (origin) with a vertex $t(\gamma)$ (terminus). A path is
called \emph{reduced} if it does not contain a back-tracking, that is  a path
of the form $ee^{-1}$, where $e\in E\Delta$.

If $\gamma=e_1,\dots, e_n$ is an edge-path in $\Delta$, where $e_i\in E\Delta$, we call $n$ the \emph{simplicial length} of $\gamma$ and denote it
by $|\gamma|$.

We denote by $\mathcal P(\Delta)$ the set of all finite reduced edge-paths
in $\Delta$.
For a vertex $x\in V\Delta$, we denote by $\mathcal P_x(\Delta)$ the
collection of all $\gamma\in \mathcal P(\Delta)$ that begin with $x$.

\end{conv}

\begin{defn}[Marking or Simplicial chart]
  Let $\Gamma$ be a finite connected graph without degree-one and
  degree-two vertices
  such that $\pi_1(\Gamma)\cong F$. Let $\alpha:F\to \pi_1(\Gamma,p)$
  be an isomorphism, where $p$ is a vertex of $\Gamma$. We call such
  $\alpha$ a \emph{simplicial chart} or a \emph{marking} for $F$.
\end{defn}

\begin{defn}[Marked metric graph structure]\label{defn:mmgs}
  Let $F$ be a free group of finite rank $k\ge 2$.  A \emph{marked
    (semi-)metric graph structure} on $F$ is a pair $(\alpha, \mathcal L$),
  where $\alpha: F\to \pi_1(\Gamma,p)$ is a simplicial chart for $F$
  and $\mathcal L$ is a (semi-)metric structure on $\Gamma$.
(A marked semi-metric graph structure $(\alpha, \mathcal L)$ is
  \emph{non-degenerate} if $\mathcal L$ is nondegenerate.)
\end{defn}

\begin{conv}\label{conv:mgs}
Let $(\alpha, \mathcal L)$ be a
marked metric graph structure on $F$. Then $(\alpha, \mathcal L)$
defines a point $T\in cv(F)$ as follows. Topologically, let $T=\widetilde
\Gamma$, with an action of $F$ on $T$ via $\alpha$. We lift the metric
structure $\mathcal L$ from $\Gamma$ to $T$ by giving every edge in
$T$ the same length as that of its projection in $\Gamma$. This makes
$T$ into an $\mathbb R$-tree equipped with a minimal free and discrete
isometric action of $F$. Thus $T\in cv(F)$ and in this situation we
will sometimes use the  notation $T=(\alpha, \mathcal L)\in cv(F)$.
Note that $T/F=\Gamma$. Moreover, it is not hard to see that every
point of $cv(F)$ arises in this fashion and that $CV(F)$ is exactly
the set of all those $T=(\alpha, \mathcal L)\in cv(F)$ where $(\alpha,
\mathcal L)$ is a marked metric graph structure on $F$ with
$vol(\mathcal L)=1$.

Note also that any nondegenerate marked semi-metric graph structure on $F$ also defines a point in $cv(F)$ by \
first contracting the edges of $\mathcal L$-length $0$, and then proceeding as above. 
\end{conv}

\begin{defn}[Elementary charts]\label{defn:ec}
  Let $\alpha: F\to \pi_1(\Gamma,p)$ be a simplicial chart for $F$.
  Fix an orientation $E\Gamma=E^+\Gamma\cup E^-\Gamma$ on $\Gamma$ and
  let $E^+\Gamma=\{e_1,\dots, e_m\}$, where $m=\#E^+\Gamma$.

  Let $V_\alpha\subseteq cv(F)$ be the set of all $T=(\alpha, \mathcal L)$ where
  $\mathcal L$ is a nondegenerate semi-metric structure on $\Gamma$. Let
  $U_\alpha$ be the set of all $T=(\alpha, \mathcal L)$ where
  $\mathcal L$ is a metric structure on $\Gamma$. Thus
  $U_\alpha\subseteq V_\alpha$. We call $V_\alpha$ the
  \emph{elementary chart} corresponding to $\alpha$ and we call
  $U_\alpha$ the \emph{elementary open chart} corresponding to
  $\alpha$.

  There is a natural map $\lambda_\alpha: V_\alpha\to \mathbb R^m$
  defined as $\lambda_\alpha (\alpha,\mathcal L)=(\mathcal
  L(e_1),\dots, \mathcal L(e_m))$. It is known that
$\lambda_\alpha: V_\alpha\to \mathbb R^m$ is injective and is a homeomorphism onto its
image. In particular, $\lambda_\alpha(U_\alpha)$ is the positive open
cone in $\mathbb R^m$, that is,  $\lambda_\alpha(U_\alpha)$ consists of all points in $\mathbb R^m$ all of
whose coordinates are positive. Therefore $U_\alpha$ is homeomorphic to an open cone in $\mathbb R^m$.
\end{defn}

The space $cv(F)$ is the union of open cones $U_\alpha$ taken
over all simplicial charts $\alpha$ on $F$. Moreover, every point
$T\in cv(F)$ belongs to only finitely many of the elementary charts
$V_\alpha$. It is also known that the standard topology on $cv(F)$
coincides with the weakest topology for which all the maps
$\lambda_\alpha^{-1}:  \lambda_\alpha(V_\alpha)\to V_\alpha$ are continuous.

Let $\alpha:F\to \pi_1(\Gamma,p)$ be a simplicial chart, let
$T=\widetilde \Gamma$ and let $j:T\to \Gamma$ denote the covering. It is
easy to see that for $\mu_n, \mu\in Curr(F)$ we have
$\displaystyle\lim_{n\to\infty}\mu_n=\mu$ if and only if
$\displaystyle\lim_{n\to\infty}\mu_n(Cyl_{\gamma})=\mu(Cyl_{\gamma})$
for every $\gamma \in \mathcal P(T)$.  Moreover, for $\mu,\mu'\in
Curr(F)$ we have $\mu=\mu'$ if and only if
$\mu(Cyl_{\gamma})=\mu'(Cyl_{\gamma})$ for every $\gamma\in \mathcal
P(T)$.

Note that for any $f\in F$ and $\gamma\in \mathcal P(T)$ we have $f
  Cyl_{\gamma}=Cyl_{f \gamma}$.  Since geodesic currents
  are, by definition, $F$-invariant, for a geodesic current $\mu$ and
  for $\gamma\in \mathcal P(T)$ the value $\mu(Cyl_{\gamma})$
  only depends on the label $j(\gamma)\in \mathcal P(\Gamma)$ of
  $\gamma$.

\begin{notation}\label{not:sp}
  For this reason, for any reduced edge-path $v\in \mathcal P(\Gamma)$,
  we denote by $\langle v,\mu\rangle_\alpha$ the value
  $\mu(Cyl_{\gamma})$ where $\gamma\in \mathcal P(T)$ is any reduced
  edge-path with label $v$.
\end{notation}

We finish this Section with the following basic lemma, needed later, which shows that for $T,T'\in cv(F)$ extremal distortions of the translation length functions for $T$ and $T'$ give the optimal stretching constants for  $F$-equivariant quasi-isometries between $T$ and $T'$.

\begin{lem}\label{lem:stretch}
Let $T,T'\in cv(F)$ and let $\phi:T\to T'$ be an $F$-equivariant
quasi-isometry.
Then the following hold:
\begin{enumerate}
\item \[
\lambda_1:=\inf_{g\in F\setminus\{1\}}
\frac{||g||_{T'}}{||g||_T}=\liminf_{d_T(x,y)\to\infty} \frac{d_{T'}(\phi(x),\phi(y))}{d_T(x,y)}.
\]
\item \[
\lambda_2:=\sup_{g\in F\setminus\{1\}}
\frac{||g||_{T'}}{||g||_T}=\limsup_{d_T(x,y)\to\infty} \frac{d_{T'}(\phi(x),\phi(y))}{d_T(x,y)}.
\]
\item There is $C>0$ such that for any $x,y\in T$ we have
\[
\lambda_1 d_T(x,y)-C\le d_{T'}(\phi(x),\phi(y))\le \lambda_2 d_T(x,y)+C .
\]
\end{enumerate}
\end{lem}

\begin{proof}
Let $x_0\in T$ be a base-point and let $x_0'=\phi(x_0)\in T'$.  Thus $\phi(gx_0)=gx_0'$
for all $g\in F$.

Let us show Part (2). For any $g\in F\setminus\{1\}$ we have
\[
\lim_{n\to\infty} \frac{d_T(x_0,g^nx_0)}{n}=||g||_T, \quad \lim_{n\to\infty} \frac{d_{T'}(x_0',g^nx_0')}{n}=||g||_{T'}
\]
and hence
\[
\lim_{n\to\infty} \frac{d_{T'}(x_0',g^nx_0')}{d_T(x_0,g^nx_0)}=\frac{||g||_{T'}}{||g||_T}.
\]
Then $\frac{||g||_{T'}}{||g||_T}\le \limsup_{d_T(x,y)\to\infty}
\frac{d_{T'}(\phi(x),\phi(y))}{d_T(x,y)}$ and so

\[
\sup_{g\in F\setminus\{1\}}
\frac{||g||_{T'}}{||g||_T}\le \limsup_{d_T(x,y)\to\infty} \frac{d_{T'}(\phi(x),\phi(y))}{d_T(x,y)}.
\]

Suppose now that $d_T(p_n,q_n)\to\infty$ and
\[
\lim_{n\to\infty} \frac{d_{T'}(\phi(p_n),\phi(q_n))}{d_T(p_n,q_n)}= \limsup_{d_T(x,y)\to\infty} \frac{d_{T'}(\phi(x),\phi(y))}{d_T(x,y)}.
\]

There is a constant $M=M(T')\ge 1$ such that
there are some $g_n,h_n\in F$  with $d_T(p_n,g_nx_0),
d_T(q_n,h_nx_0)\le M$ and such that the geodesic $[g_nx_0,h_nx_0]$
projects to a closed cyclically reduced path in $T/F$. That is, the
points $g_nx_0,h_nx_0$ belong to the axis of the element $h_ng_n^{-1}$
and $d_T(g_nx_0,h_nx_0)=||h_ng_n^{-1}||_T$. Translating $[p_n,q_n]$ by
$g_n^{-1}$ we may assume that $g_n=1$ for every $n\ge 1$. Thus
$d(x_0,h_nx_0)=||h_n||_T$ and $d_T(p_n,x_0), d_T(q_n,h_nx_0)\le M$.
Hence $| d_T(p_n,q_n)-||h_n||_T|\le 2M$. Note that $x_0$ and $h_nx_0$ belong to the axis of $h_n$ in $T$.

Denote $p_n'=\phi(p_n)$, $q_n'=\phi(q_n)$.
Since $\phi$ is an $F$-equivariant quasi-isometry, the $\phi$-image of the axis of $h_n$ in $T$ is an $h_n$-invariant quasigeodesic in $T'$ which is at a bounded Hausdorff distance from an $h_n$-invariant geodesic in $T'$, that is, from the axis of $h_n$ in $T'$. Hence
there is some
constant $C\ge 1$ such that $| d_{T'}(p_n',q_n')-||h_n||_{T'}|\le C$.

Therefore
\[
\lim_{n\to\infty} \frac{d_{T'}(\phi(p_n),\phi(q_n))}{d_T(p_n,q_n)}\le
\limsup_{n\to\infty} \frac{||h_n||_T+M}{||h_n||_{T'}-C}\le
\sup_{g\in F\setminus\{1\}}
\frac{||g||_{T'}}{||g||_T}.
\]
Hence
\[
\sup_{g\in F\setminus\{1\}}
\frac{||g||_{T'}}{||g||_T}=\limsup_{d_T(x,y)\to\infty} \frac{d_{T'}(\phi(x),\phi(y))}{d_T(x,y)},
\]
as required.

Part (1) is established using a similar argument to part (2) and we
omit the details.

For part (3), let $x,y\in T$ be arbitrary and let $x'=\phi(x)$, $y'=\phi(y)$.
As in the proof of (2), there exist elements $g,h\in F$ such that $d_T(x,gx_0),
d_T(y,hx_0)\le M$ and such that the geodesic $[gx_0,hx_0]$
projects to a closed cyclically reduced path in $T/F$. By exactly the same argument as above we deduce
\begin{gather*}
d_{T'}(x',y')\le ||h||_{T'}+C\le \lambda_2||h||_T+C\le \\
\lambda_2 (d_T(x,y)+2M)+C=\lambda_2 d_T(x,y)+(2M\lambda+C),
\end{gather*}
as required.
The proof that $d_{T'}(x',y')\ge\lambda_1 d_T(x,y)-(2M\lambda+C)$, is similar, and we omit the details.

\end{proof}

\section{Geometric entropy of a geodesic current}\label{sect:ge}

\begin{defn}[Geometric entropy of a current]\label{defn:ge}
Let $\mu\in Curr(F)$ and let $T\in cv(F)$.
Define the \emph{geometric entropy of $\mu$ with respect to $T$}  as
\[
h_T(\mu):=\liminf_{\overset{d_T(x,y)\to\infty}{x,y\in T}} \frac{-\log \mu (Cyl_{[x,y]})
}{d_T(x,y)}.
\]
\end{defn}

If $\mu (Cyl_{[x,y]})=0$, we interpret $\log 0=-\infty$. Thus for $\mu=0$ and any $T\in cv(F)$ we have $h_T(\mu)=\infty$. For any $\mu\in Curr(F)$, $\mu\ne 0$ and any $T\in cv(F)$ we have $0\le h_T(\mu)<\infty$.
Informally, $h_{T}(\mu)$ measures the slowest exponential decay rate
(with respect to $d_T(x,y)$)
of the "weights" $\mu (Cyl_{[x,y]})$ as $d_T(x,y)\to\infty$. Thus, taking into account that $T$ is a discrete $\mathbb R$-tree, we see that if $h_{T}(\mu)>s> 0$ then there exists $C>0$ such
that for every $x,y\in T$ with $d_T(x,y)\ge 1$ we have
\[
\mu (Cyl_{[x,y]}) \le C\ \exp(-s\, d_T(x,y)).
\]

The following is a more combinatorial interpretation of the geometric entropy that follows immediately from unraveling the definitions (we use
notation~\ref{not:sp}).

\begin{prop}
Let $\mu\in Curr(F)$, $\mu\ne 0$ and let $T\in cv(F)$ be determined by
the pair $(\alpha, \mathcal L)$, where $\alpha: F\to\pi_1(\Gamma,p)$ is an isomorphism
and $\mathcal L$ is a metric graph structure on $\Gamma$.
Then
\begin{enumerate}
\item
$$h_T(\mu):=\liminf_{\overset{|v|\to\infty}{v\in \mathcal P(\Gamma)}} \frac{-\log \langle v, \mu\rangle_\alpha}{\mathcal L(v)}=\liminf_{\overset{\mathcal L(v)\to\infty}{v\in \mathcal P(\Gamma)}} \frac{-\log \langle v, \mu\rangle_\alpha
}{\mathcal L(v)} ;$$

\item if $h_{T}(\mu)>s\ge 0$ then there exists $C>0$ such
that, for every $v\in \mathcal P(\Gamma)$
\[
\langle v,\mu\rangle_\alpha \le C\ \exp(-s \mathcal L(v)).
\]
\end{enumerate}
\end{prop}
We already noted that if $\mu=0$, we get $h_T(\mu)=\infty$ for any $T\in cv(F)$. On the
other hand, there is a family of so-called rational currents $\{\eta_g\}_ {g\in F\setminus\{1\}}$, such that for any $T\in cv(F)$, $h_T(\eta_g)=0$.

The current $\eta_g$ is defined as follows: for every Borel subset $S\subseteq\partial F^2$, $\eta_g(S)$ is equal to the number of $F$-translates of the bi-infinite geodesic $(g^{-\infty},g^{\infty})$ that belong to $S$.  

The following statement is an immediate corollary of the definition of
geometric entropy:
\begin{prop}
Let $\mu\in Curr(F), \mu\ne 0$ and let $T\in cv(F)$.
\begin{enumerate}
\item For any $c>0$ we have
$h_T(\mu)=h_T(c\mu)$.
\item For any $c>0$ we have $h_{cT}(\mu)= \frac{1}{c}h_T(\mu)$.
\end{enumerate}
\end{prop}
Thus $h_T(\mu)$ depends only on the projective class $[\mu]\in \mathbb
PCurr(F)$ of $\mu$.

\begin{lem}\label{lem:qi}
Let $T,T'\in cv(F)$. Let $\phi: T\to T'$ be an $F$-equivariant quasi-isometry.
There exists an integer $M=M(\phi)\ge 1$ with the following  property.

Let $x_1,x_2\in T$ and let $x_1'=\phi(x_1), x_2'=\phi(x_2)$. Let
  $y_1,y_2\in [x_1',x_2']$ be such that
  $d_{T'}(x_1',y_1)=d_{T'}(x_2',y_2)=M$.
Then
  \[\phi(Cyl_{[x_1,x_2]})\subseteq Cyl_{[y_1,y_2]}.\]
\end{lem}

\begin{proof}
This statement is a straightforward consequence of the "Bounded Cancellation Lemma"~\cite{Coo} for quasi-isometries between Gromov-hyperbolic spaces.
Indeed, let $[x,y]\subseteq T$ be a geodesic segment in $T$ and let $\gamma$ be a geodesic ray in $T$ with initial point $y$, such that the path $[x,y]\gamma$ is a geodesic (that is, there is no cancellation between $[x,y]$ and $\gamma$. There is a unique geodesic ray $\gamma'$   in $T'$ starting at $\phi(y)$ such that $\gamma'$ is at a finite Hausdorff distance from $\phi(\gamma)$. The "Bounded Cancellation Lemma"  implies that the cancellation between
$[\phi(x),\phi(y)]$ and $\gamma'$ is bounded by some constant $M\ge 1$ depending only on the quasi-isometry $\phi$.
It is not hard to check that this constant $M$ satisfies the
requirements of the proposition and we leave the details to the reader.
\end{proof}

\begin{prop}\label{prop:qi1}
Let $T,T'\in cv(F)$. Let $\phi: T\to T'$ be an $F$-equivariant quasi-isometry.
There exists an integer $M_1=M_1(\phi)\ge 1$ with the following  property.

Let $x_1,x_2\in T$ and let $x_1'=\phi(x_1), x_2'=\phi(x_2)$.

Then there exist points $p_1,p_2, q_1,q_2\in T'$ such that
$[p_1,p_2]=[x_1',x_2']\cap [q_1,q_2]$, such that
$d(q_i,x_i')\le M_1$ and such that
\[
Cyl_{[q_1,q_2]}\subseteq \phi(Cyl_{[x_1,x_2]}).
\]
\end{prop}
\begin{proof}
Let $\psi: T'\to T$ be an $F$-equivariant quasi-isometry which is a quasi-inverse of $\phi$.
Let $M=M(\psi)>0$ be the constant provided by Lemma~\ref{lem:qi} for $\psi$.

Choose $\xi,\zeta\in \partial T$ such that the bi-infinite geodesic $[x_1,x_2]\subseteq [\xi,\zeta]$.
Then $x_1'=\phi(x_1)$ and $x_2'=\phi(x_2)$ are at distance at most $C_1=C_1(\phi)$ from $[\phi(\xi), \phi(\zeta)]$.
Thus $[\phi(\xi), \phi(\zeta)]\cap [x_1',x_2']=[p_1,p_2]$, where  $d(x_i',p_i)\le C_1$.

Recall that $\psi$ is a quasi-isometry that is a quasi-inverse of $\phi$, so that $\psi(\phi(\xi))=\xi$ and $\psi(\phi(\zeta))=\zeta$.
Note also that $[x_1,x_2]\subseteq [\xi,\zeta]=[\psi(\phi(\xi)),\psi(\phi(\zeta))]$ and that $[p_1,p_2]\subseteq
[\phi(\xi),\phi(\zeta)]$.
Therefore there is some $C_2=C_2(\phi)>0$ such that if $q_1\in [\phi(\xi), p_1]$, $q_2\in [p_2, \phi(\zeta)]$ are such that
$d(q_i,p_i)\ge C_2$ then $[x_1,x_2]\subseteq [\psi(q_1),\psi(q_2)]$ and $d(x_i,\psi(q_i))\ge M=M(\psi)$.

Choose $q_1\in [\phi(\xi), p_1]$, $q_2\in [p_2, \phi(\zeta)]$ such that $d(q_i,p_i)=C_2$. Thus $[x_1,x_2]\subseteq
[\psi(q_1),\psi(q_2)]$ and $d(x_i,\psi(q_i))\ge M$. Note that $d(x_i,\psi(q_i))\le C_3=C_3(\psi)$  since
$\psi$ is a quasi-isometry and $d(x_i',q_i)\le C_2+C_1$.

Thus $[x_1,x_2]\subseteq [\psi(q_1),\psi(q_2)]$ and $d(x_i,\psi(q_i))\ge M=M(\psi)$.
Then by Lemma~\ref{lem:qi}, applied to $\psi$, we have
\[
\psi(Cyl_{[q_1,q_2]})\subseteq Cyl_{[x_1,x_2]}.
\]
Applying $\phi$, we obtain
\[
Cyl_{[q_1,q_2]}\subseteq \phi(Cyl_{[x_1,x_2]}).
\]
Note that by construction $[p_1,p_2]=[x_1',x_2']\cap [q_1,q_2]$. Moreover,
$d(q_i,x_i')\le C_1+C_2$. Thus all the requirements of the proposition are satisfied, which completes the proof.
\end{proof}

\begin{cor}\label{cor:0}
Let $\mu\in Curr(F)$ and let $T,T'\in cv(F)$. Then
\[
h_{T}(\mu)>0 \iff h_{T'}(\mu)>0.
\]
\end{cor}
\begin{proof}
Suppose that $h_{T'}(\mu)>0$. Hence there exists $s>0$, $C>0$ such
that for every $x,y\in T'$, $x\ne y$, we have
\[
\mu(Cyl_{[x,y]}) \le C \exp( -s\, d_{T'}(x,y)).
\]
Let $\phi:T\to T'$ be a $F$-equivariant
$(\lambda,\lambda)$-quasi-isometry, where $\lambda\ge 1$ (i.e. a quasi-isometry with both the multiplicative and the additive constants equal to 

$\lambda$).
Let $M,C>0$ be provided by Lemma~\ref{lem:qi}. Let $x_1,x_2\in
T$ be such that $N:=d_{T}(x_1,x_2)>20\lambda^2M$. Let $x_i'=\phi(x_i)\in
T'$, $i=1,2$. Thus $d_{T'}(x_1',x_2')\ge N/\lambda-\lambda\ge N/{2\lambda}$.
Let $y_1,y_2\in [x_1',x_2']$ be such that
$d_{T'}(x_1',y_1)=d_{X'}(x_2',y_2)=M$. Thus
\[
d_{T'}(y_1,y_2)\ge \frac{N}{2\lambda}-2M\ge \frac{N}{3\lambda}=\frac{d_T(x_1,x_2)}{3\lambda}.
\]

By Lemma~\ref{lem:qi} we have:
\[
\phi(Cyl_{[x_1,x_2]})\subseteq Cyl_{[y_1,y_2]}
\]

Recall that under the identifications  $\partial q_T:\partial T\rightarrow \partial
F$ and  $\partial q_{T'}:\partial T'\rightarrow\partial F$, for any $S\subseteq \partial T$, $\partial q_T(S) = \partial q_{T'}(\phi(S))\subseteq\partial F$.
Therefore
\begin{gather*}
\mu(Cyl_{[x_1,x_2]})\le \mu(Cyl_{[y_1,y_2]})\le C \exp( -s
d_{T'}(y_1,y_2))\le\\
\le C \exp(-s\frac{d_T(x_1,x_2)}{3\lambda}).
\end{gather*}
This implies that $h_T(\mu)\ge \frac{s}{3\lambda}>0$, as required.
\end{proof}

Corollary~\ref{cor:0} implies that the following notion is well-defined.
We say that a current $\mu\in Curr(F)$ has \emph{exponential decay}
or \emph{decays exponentially fast} if for some (equivalently, for any) $T\in cv(F)$ we have $h_T(\mu)>0$.
Similarly, we say that $\mu$ has \emph{subexponential decay}
 if for some (equivalently, for any) $T\in cv(F)$ we have $h_T(\mu)=0$.

Recall that for a graph $\Gamma$ we denote by $\mathcal P(\Gamma)$ the set of all nontrivial reduced finite edge-paths in $\Gamma$.

\begin{prop}\label{prop:cont}
Let $\mu\in Curr(F), \mu\ne 0$. Then the function
\[
E_\mu: cv(F)\to \mathbb R, \quad T\mapsto h_T(\mu)
\]
is continuous.
\end{prop}

\begin{proof}
It suffices to check that for every simplicial chart $\alpha:F\to
\pi_1(\Gamma,p)$ on $F$ the restriction of $E_\mu$ to the elementary
chart $V_\alpha\subseteq cv(F)$ is continuous.

Let $E\Gamma=E^+\Gamma\sqcup E^-\Gamma$ be an orientation on $\Gamma$, let $m=\#E^+\Gamma$ and let $E^+\Gamma=\{e_1,\dots, e_m\}$.
Recall that $V_\alpha$ consists of all points of the form $(\alpha,
\mathcal L)$, where $\mathcal L$ is a nondegenerate semi-metric
structure on $\Gamma$.

Let $T=(\alpha, \mathcal L)\in U_{\alpha}$. Let $\epsilon>0$ be arbitrary. Choose $\epsilon_1>0$ such that $\epsilon_1 h_T(\mu)<\epsilon$.
Then  there exists a neighborhood $\Omega$ of $T$ in $V_\alpha$ with the following property. For any $v\in \mathcal P(\Gamma)$ and for any $T'=(\alpha, \mathcal L')\in \Omega$ we have
\[
 (1-\epsilon_1) \mathcal L'(v)\le  \mathcal L(v)\le (1+\epsilon_1) \mathcal L'(v).
\]
Therefore 
\[
\frac{-\log \langle v, \mu\rangle_\alpha}{\mathcal L(v)}(1-\epsilon_1)\le \frac{-\log \langle v, \mu\rangle_\alpha}{\mathcal L'(v)}\le \frac{-\log \langle v, \mu\rangle_\alpha}{\mathcal L(v)}(1+\epsilon_1).
\]
It follows that
\[
h_T(\mu)(1-\epsilon_1)\le h_{T'}(\mu)\le h_T(\mu)(1+\epsilon_1)
\]
and hence
\[
|E_\mu(T)-E_\mu(T')|=|h_{T'}(\mu)-h_T(\mu)|\le h_\mu(T)\epsilon_1\le \epsilon.
\]

Thus $E_\mu$ is continuous at the point $T$ and, since $T\in V_\alpha$ was arbitrary, $E_\mu$ is continuous on $V_\alpha$, as required.
\end{proof}

\begin{notation}\label{notation:support}
Let $\alpha: F\to \pi_1(\Gamma)$ be a simplicial chart and let $\mu\in Curr(F)$ be a geodesic current. Define
\[
supp_\alpha(\mu):=\{v\in \mathcal P(\Gamma): \langle v,\mu\rangle_\alpha>0\}.
\]
We refer to $supp_\alpha(\mu)$ as the \emph{support of $\mu$ with respect to $\alpha$}.

Let $\mathcal L$ be a metric graph structure on $\Gamma$ and let
$T_{(\alpha, \mathcal L)}\in cv(F)$ correspond to the universal cover of $\Gamma$ with the edge-length lifted from $\mathcal L$.
Define
\[
\rho(T_{(\alpha,\mathcal L)},\mu):=\liminf_{R\to\infty} \frac{\log\beta(R)}{R}
\]
where $\beta(R)=\#\{v\in supp_\alpha(\mu): \mathcal L(v)\le R\}$.
Thus $\rho(T_{(\alpha,\mathcal L)})$ measures the exponential growth rate of the support $supp_\alpha(\mu)$ with respect to the metric structure $\mathcal L$.

Let now $\mu\in Curr(F)$, $\mu\ne 0$ and let $T\in cv(F)$ be arbitrary. There exists an (essentially unique) marked metric graph structure $(\alpha,\mathcal L)$ such that $T=T_{(\alpha,\mathcal L)}$. Put $\rho(T,\mu):=\rho(T_{(\alpha,\mathcal L)})$.
\end{notation}

It is not hard to show that if $\rho(T_0,\mu)=0$ for some $T_0\in cv(F)$ then $\rho(T,\mu)=0$ for every $T\in cv(F)$.

\begin{thm}\label{thm:support}
Let $T\in cv(F)$ and $\mu\in Curr(F)$, $\mu\ne 0$.
Then
\[
h_T(\mu)\le \rho(T,\mu)
\]
\end{thm}
\begin{proof}
We may assume that $T=T_{(\alpha,\mathcal L)}$ for some marked metric graph structure $(\alpha:F\to\pi_1(\Gamma),\mathcal L)$ on $F$.
If $h_T(\mu)=0$, there is nothing to prove. Suppose now that $h_T(\mu)>0$. Let $s>0$ be arbitrary such that $s<h_T(\mu)$. Thus there exists $n_0\ge 1$ such that for any $p,q\in T$ with $d_T(p,q)\ge n_0$ we have
$\mu(Cyl_{[p,q]})\le \exp(-sd_T(p,q))$.

Since $\mu\ne 0$, there exists an edge $e$ of $\Gamma$ such that $\langle e,\mu\rangle_\alpha>0$, that is $e\in supp_\alpha(\mu)$. Let $[x,y]\subseteq T$ be a lift of $e$ to $T$. Thus
$\mu(Cyl_{[x,y]})=\langle e,\mu\rangle_\alpha>0$.
For any integer $n\ge n_0$ let $[x,z_1],\dots [x,z_m]$ be all geodesic segments of length $n$ in $T$ that contain $[x,y]$ as an initial segment and such that
$\mu(Cyl_{[x,p_i]})>0$, where $p_i=z_i$  if $z_i$ is a vertex of $T$, whereas if $z_i$ is an interior point of an edge, $p_i$ is the endpoint of that edge further away from $x$ than $z_i$. 
Then $\mu(Cyl_{[x,p_i]})\le \exp(-sn)$. Moreover, the cylinders $Cyl_{[x,p_1]},\dots, Cyl_{[x,p_m]}$ are disjoint and
$\mu(Cyl_{[x,y]})=\sum_{i=1}^m \mu(Cyl_{[x,p_i]})$. Let $v_i\in \mathcal P(\Gamma)$ be the path in $\Gamma$ which labels the segment $[x,p_i]$ in $T$. Note that by construction $n\le \mathcal L(v_i)\le n+c$, where $c>0$ is the length of the longest edge in $(\Gamma,\mathcal L)$. Thus, again, by construction, $m\le \beta(n+c)$. We have
\[
0<\langle e,\mu\rangle_\alpha=\sum_{i=1}^m \langle v_i,\mu\rangle_\alpha\le m \exp(-sn)\le \beta(n+c)\exp(-sn).
\]
It follows that $\rho(T,\mu)\ge s$ since if $\rho(T,\mu)<s$ then $\beta(n+c)\exp(-sn)\to 0$ as $n\to\infty$, contradicting the fact that $\langle e,\mu\rangle_\alpha>0$. Thus for every $0<s<h_T(\mu)$ we have $\rho(T,\mu)\ge s$. Therefore $\rho(T,\mu)\ge h_T(\mu)$, as required.
\end{proof}
\begin{rem}
Suppose $T=T_{(\alpha,\mathcal L)}$ where for every edge $e$ of $\Gamma$ we have $\mathcal L(e)=1$.
Let $\Omega(\Gamma,\mu)$ be the set of all semi-infinite reduced edge paths $\gamma=e_1,\dots, e_n, \dots $ in $\Gamma$ such that for every finite subpath $v$ of $\gamma$ we have $v\in supp_\alpha(\mu)$. We can think of $\Omega(\Gamma,\mu)$ as a subset of the set $A^\omega$ of all semi-infinite words in the alphabet $A$ consisting of all oriented edges of $\Gamma$. The subset $\Omega(\Gamma,\mu)\subseteq A^\omega$ is clearly invariant under the shift map $\sigma: A^\omega\to A^\omega$ which erases the first letter of a semi-infinite word from $A^\omega$. Thus $\Omega(\Gamma,\mu)$ is a subshift (not necessarily of finite type) of the full shift $(A^\omega,\sigma)$. It is easy to see from the definitions that in this case $\rho(T_{(\alpha,\mathcal L)})$ is equal to the topological entropy $h_{topol}(\Omega(\Gamma,\mu))$ of the subshift $\Omega(\Gamma,\mu)$ of the set $A^\omega$. In this situation Theorem~\ref{thm:support} says that $h_T(\mu)\le h_{topol}(\Omega(\Gamma,\mu))$.

As we already observed, it is not hard to show that if $\rho(T_0,\mu)=0$ for some $T_0\in cv(F)$ then $\rho(T,\mu)=0$ for every $T\in cv(F)$. In this case Theorem~\ref{thm:support} implies $h_T(\mu)=0$ for every $T\in cv(F)$.
\end{rem}

\section{Tame currents}

\begin{defn}[Tame current with respect to a tree]\label{defn:tameT}
Let $\mu\in Curr(F)$ and $T\in cv(F)$. We say that $\mu$ is \emph{tame
  with respect to $T$} or \emph{$T$-tame} if for every $M\ge 1$ there is $C=C(M)\ge 1$
such that whenever $a_1,b_1,a_2,b_2\in T$ satisfy $d(a_1,a_2)\le M,
d(b_1,b_2)\le M$, $a_1\ne b_1$, $a_2\ne b_2$ then
\[
\frac{1}{C}\mu ( Cyl_{[a_2,b_2]})   \le \mu ( Cyl_{[a_1,b_1]})\le C \mu ( Cyl_{[a_2,b_2]}).
\]
We call $C=C(M)$ the \emph{$T$-tameness constant} corresponding to $M$.
\end{defn}

\begin{lem}\label{lem:wtame}
Let $T\in cv(F)$ and $\mu\in Curr(F)$. Suppose that there is some $N\ge 1$ such that for every $M\ge N$ there exists $D=D(M)\ge 1$ such that whenever
$[x,y]\subseteq [a,b]$ where $x\ne y$ and $d_T(a,x)=d_T(y,b)=M$ then
\[
\mu ( Cyl_{[x,y]})\le D \mu ( Cyl_{[a,b]}).
\]
Then $\mu$ is $T$-tame.
\end{lem}
\begin{proof}

Suppose that for every $M\ge N$ there is $D=D(M)\ge 1$ as in Lemma~\ref{lem:wtame}.

We need to prove that $\mu$ is $T$-tame. It is easy to see that it suffices to verify
the conditions of Definition~\ref{defn:tameT} for all sufficiently large $M$.

Let $M\ge N$ be arbitrary.
Suppose now that $[x,y]\subseteq [a,b]$ with $d_T(a,x)=d_T(y,b)\le M$ and  $x\ne y$.
Choose a geodesic segment $[a',b']$ in $T$ such that $[a,b]\subseteq [a',b']$ and such that $d_T(a',x)=d_T(b',y)=M$.
Then $Cyl_{[a',b']}\subseteq Cyl_{[a,b]}\subseteq Cyl_{[x,y]}$. Hence by assumption on $D=D(M)$ we have
\[
\mu(Cyl_{[a,b]})\le \mu(Cyl_{[x,y]})\le D \mu(Cyl_{[a',b']})\le D \mu(Cyl_{[a,b]})
\]
so that
\[
\mu(Cyl_{[a,b]})\le \mu(Cyl_{[x,y]})\le D \mu(Cyl_{[a,b]}).\tag{\dag}
\]
Thus $(\dag)$ holds whenever $[x,y]\subseteq [a,b]$ with $d_T(a,x)=d_T(y,b)\le M$.

Suppose now that $M\ge N$ and $a_1,b_1,a_2,b_2\in T$ satisfy $d(a_1,a_2)\le M,
d(b_1,b_2)\le M$. We may assume that $d(a_1,b_1)\ge 3M$ since otherwise the requirements of
Definition~\ref{defn:tameT} are easily satisfied.
(Indeed, if $d(a_1,b_1)\le 3M$, then the points $a_1,a_2,b_1,b_2$ lie in an $F$-translated of a
fixed closed ball of radius $6M$, which is a finite subtree of $T$).
Then $[a_1,b_1]\cap [a_2,b_2]$ is a non-degenerate geodesic segment.
Put $[x,y]=[a_1,b_1]\cap [a_2,b_2]$.

Then
\[
d(x,a_1), d(x,a_2), d(y,b_1),d(y,b_2)\le M
\]
and $[x,y]\subseteq [a_1,b_1]$, $[x,y]\subseteq [a_2,b_2]$.
Then by $(\dag)$ we have
\[
\mu(Cyl_{[a_1,b_1]})\le \mu(Cyl_{[x,y]})\le D \mu(Cyl_{[a_2,b_2]})
\]
and
\[
\mu(Cyl_{[a_2,b_2]})\le \mu(Cyl_{[x,y]})\le D \mu(Cyl_{[a_1,b_1]}).
\]
Therefore $\mu$ is $T$-tame with as required.
\end{proof}

\begin{prop}\label{prop:tame}
Let $\mu\in Curr(F)$ and let $T,T'\in cv(F)$.
Then $\mu$ is tame with respect to $T$ if and only if $\mu$ is tame
with respect to $T'$.
\end{prop}

Proposition~\ref{prop:tame} implies that the following notion is well-defined and does not depend on the choice of $T\in cv(F)$:

\begin{defn}[Tame current]
Let $\mu\in Curr(F)$. We say that $\mu$ is \emph{tame} if for some (equivalently, for any) $T\in cv(F)$ the current $\mu$ is $T$-tame.
\end{defn}

\begin{proof}
Suppose that $\mu$ is tame with respect to $T'$.

Let $x\in T$ and $x'\in T'$ be arbitrary vertices.
Let $\phi:T\to T'$ be an $F$-equivariant
$(\lambda,\lambda)$-quasi-isometry such that $\phi(x)=x'$. Let $M=M(\phi)\ge
1$ be the constant provided by Lemma~\ref{lem:qi}.

We need to prove that $\mu$ is tame with respect to $T$. By Lemma~\ref{lem:wtame} it suffices to show that the conditions of
Lemma~\ref{lem:wtame} hold for $\mu$.

Let $M_0\ge 1$ be sufficiently big (to be specified later) and suppose $s,t,a,b\in T$ are such that $[s,t]\subseteq [a,b]$
with $d_T(s,a)=d_T(t,b)=M_0$.

Let $s',t'\in [\phi(a),\phi(b)]$ be such that
\begin{gather*}
d_{T'}(\phi(s),s')=d_{T'}(\phi(s), [\phi(a),\phi(b)]) \quad \text{  and  }\\
d_{T'}(\phi(t),t')=d_{T'}(\phi(t), [\phi(a),\phi(b)]).
\end{gather*}

Since $\phi$ is a
quasi-isometry and $T,T'$ are Gromov-hyperbolic, we have
\[d_{T'}(\phi(s),s'), d_{T'}(\phi(t),t')\le C_1,\]
where $C_1=C_1(\phi)>0$ is some constant. Note that $[\phi(a),\phi(b)]\cap [\phi(s),\phi(t)]=[s',t']$.
We may assume that $M_0$ was chosen big enough so that \[d_{T'}(\phi(a),s'), d_{T'}(\phi(b),t')\ge M_1=M_1(\psi)\] where $M_1=M_1(\psi)$
is the constant provided by Proposition~\ref{prop:qi1}.

Proposition~\ref{prop:qi1} implies that there exist $p_1,p_2,q_1,q_2\in T'$ such that $[p_1,p_2]=[q_1,q_2]\cap [\phi(a),\phi(b)]$ and
such that
\[
Cyl_{[q_1,q_2]}\subseteq \phi (Cyl_{[a,b]})
\]
and such that $d_{T'}(q_1,\phi(a)), d_{T'}(q_2,\phi(b))\le M_1$. Thus
$d_{T'}(\phi(a),s'), d_{T'}(\phi(b),t')\ge M_1=M_1(\psi)$ implies that $[s',t']\subseteq [p_1,p_2]$.

Let $s'',t''\in [s',t']\subseteq [\phi(a),\phi(b)]$ be such that
$d_{T'}(s',s'')=d_{T'}(t',t'')=M$. Thus $s'',t''\in [\phi(s),\phi(t)]$ and
\[
M\le d_{T'}(\phi(s),s'')\le M+C_1,\quad   M\le d_{T'}(\phi(t),t'')\le M+C_1.
\]
Then by Lemma~\ref{lem:qi}
\[
\phi(Cyl_{[s,t]})\subseteq Cyl_{[s'',t'']}
\]
and hence
\[
\mu(Cyl_{[s,t]} )\le \mu(Cyl_{[s'',t'']}).
\]

Note that since $\phi$ is a $(\lambda,\lambda)$-quasi-isometry and
$d_{T}(a,s)=M_0$, $d_{T}(b,t)=M_0$ then \[d_{T'}(\phi(a),\phi(s)), d_{T'}(\phi(b),\phi(t))\le \lambda M_0+\lambda.\] Since $d_{T'}(\phi(s),s'), d_{T'}(\phi(t),t')\le C_1$, it follows that \[d_{T'}(\phi(a),s'), d_{T'}(\phi(b),t')\le \lambda M_0+\lambda+C_1.\] Since $d_{T'}(s',s'')=d_{T'}(t',t'')=M$, we have \[d(\phi(a),s''), d(\phi(b),t'')\le M+\lambda M_0+\lambda+C_1.\]
Since $p_1\in [\phi(a),s'']$, $p_2\in [t'',\phi(b)]$ and $d_{T'}(q_1,\phi(a)), d_{T'}(q_2,\phi(b))\le M_1$, we get
\[
d_{T'}(s'',q_1), d_{T'}(t'',q_2)\le M_1+M+\lambda M_0+\lambda+C_1.
\]

Put $M_2=M_1+M+\lambda M_0+\lambda+C_1$.
Since $\mu$ is $T'$-tame,
\[
\mu ( Cyl_{[s'',t'']})\le C_2 \mu( Cyl_{[q_1,q_2]}),
\]
where $C_2=C(M_2)$ is the $T'$-tameness constant for $\mu$ corresponding to $M_2$.

Recall that
$Cyl_{[q_1,q_2]}\subseteq \phi (Cyl_{[a,b]})$ and therefore
\[
\mu( Cyl_{[q_1,q_2]})\le \mu (Cyl_{[a,b]}).
\]
Thus we have
\begin{gather*}
\mu(Cyl_{[s,t]})\le \mu(Cyl_{[s'',t'']})\le C_2 \mu( Cyl_{[q_1,q_2]})\le\\
\le C_2 \mu (Cyl_{[a,b]}).
\end{gather*}
Hence by Lemma~\ref{lem:wtame} $\mu$ is $T$-tame, as required.

\end{proof}

\section{The geometric entropy function on the Outer Space}

\begin{thm}\label{thm:lowerbound}
Let $T\in cv(F)$ and let $\mu\in Curr(F)$ be a tame current. Then for
any $T'\in cv(F)$ we have:
\[
h_{T'}(\mu)\ge  h_T(\mu)\inf_{f\in F\setminus\{1\}}\frac{||f||_T}{||f||_{T'}}.
\]
\end{thm}

\begin{proof}

Put $h=h_T(\mu)$.

Let $x\in T$ and $x'\in T'$ be arbitrary vertices. Let $\phi: T'\to T$
be an $F$-equivariant quasi-isometry  such that $\phi(x')=x$. Thus
$\phi(gx')=gx$ for every $g\in F$. Let $M\ge 1$ be provided by Lemma~\ref{lem:qi}.

Let $a_n,b_n\in T'$ be such that $\lim_{n\to\infty}
d_{T'}(a_n,b_n)=\infty$ and such that
\[
h_{T'}(\mu)=\lim_{n\to\infty} \frac{-\log  \mu(Cyl_{[a_n,b_n]}) }{d_{T'}(a_n,b_n)}.
\]

Note that there exists some constant $M'=M'(T')\ge 1$ (it can be taken equal to $vol(\mathcal L')$), such that for
any finite reduced edge-path $v$ in $T'/F$ there exists a cyclically reduced
closed edge-path $\widehat v$ in $T'/F$ containing $v$ or contained in $v$ (as a subpath) and
such that $|\mathcal L'(v)-\mathcal L'(\widehat v)|\le M'$.

Then there exists $h_n,g_n\in F$ such that $d_{T'}(a_n,h_nx')\le M'$,
$d_{T'}(b_n,g_nx')\le M'$ and such that $[h_nx',g_nx']$ projects to a
closed cyclically reduced path in $T'/F$.

After translating
$[a_n,b_n]$ by $h_n^{-1}$, we may assume that $h_n=1$. Thus $[x',g_nx']$
is contained in the axis of $g_n$ and $d_{T'}(x',g_nx')=||g_n||_{T'}$.
Note that $\displaystyle\lim_{n\to\infty}d_{T'}(a_n,b_n)=\infty$ implies $\displaystyle\lim_{n\to\infty}||g_n||_{T'}=\infty$.

Since $\mu$ is tame, there exists a constant $C_1\geq 1$ such that
\[
\mu(Cyl_{[a_n,b_n]})\le C_1 \mu(Cyl_{[x',g_nx']}).
\]
Therefore
\[
\frac{-\log  \mu(Cyl_{[a_n,b_n]})}{d_{T'}(a_n,b_n)}\ge \frac{-\log
  \mu(Cyl_{[x',g_nx']})-\log C_1}{d_{T'}(x',g_nx')+2M'}.
\]
and
\[
h_{T'}(\mu)\ge \liminf_{n\to\infty} \frac{-\log
  \mu(Cyl_{[x',g_nx']})}{d_{T'}(x',g_nx')}.
\]

Note that $\phi(x')=x$ and $\phi(g_nx')=g_nx$.
Moreover, since $d_{T'}(x',g_nx')=||g_n||_{T'}$ and $\phi$ is a
quasi-isometry, there is a constant $C_2>0$ independent of $n$ such that for
every $n\ge 1$
\[
\left| d_T(x,g_nx)-||g_n||_T \right|\le C_2.
\]
Also, we have $\lim_{n\to\infty} ||g_n||_T=\infty$.
Let $[y_n,z_n]\subseteq [x,g_nx]$ be such that $d_T(x,y_n)=d_T(g_nx,z_n)=M$, where $M$ is a constant provided by Lemma~\ref{lem:qi}, with
\[
 \phi \left(Cyl_{[x',g_nx']}\right)\subseteq Cyl_{[y_n,z_n]}
\]
and hence
\[
\mu ( Cyl_{[x',g_nx']}) \le \mu( Cyl_{[y_n,z_n]})\le C_3 \mu( Cyl_{[x,g_nx]}),
\]
where the constant $C_3\geq 1$ exists and the last inequality holds since $\mu$ is tame.

Thus
\begin{gather*}
\frac{-\log \mu ( Cyl_{[x',g_nx']})}{d_{T'}(x',g_nx')}\ge
\frac{-\log \mu ( Cyl_{[x,g_nx]})-\log C_2}{d_{T'}(x',g_nx')}=\\
\frac{-\log \mu ( Cyl_{[x,g_nx]})-\log
  C_2}{d_T(x,g_nx)}\cdot \frac{d_T(x,g_nx)}{d_{T'}(x',g_nx')}=\\
\frac{-\log \mu ( Cyl_{[x,g_nx]})-\log C_2}{d_T(x,g_nx)}\cdot
\frac{d_T(x,g_nx)}{||g_n||_{T'}}\ge \\
\frac{-\log \mu ( Cyl_{[x,g_nx]})-\log C_2}{d_T(x,g_nx)}\cdot \frac{||g_n||_T-C_3}{||g_n||_{T'}}
\end{gather*}

Therefore
\begin{gather*}
h_{T'}(\mu)\ge \liminf_{n\to\infty} \frac{-\log
  \mu(Cyl_{[x',g_nx']})}{d_{T'}(x',g_nx')}\ge \\
\liminf_{n\to\infty} \frac{-\log \mu ( Cyl_{[x,g_nx]})-\log
  C_2}{d_T(x,g_nx)}\cdot \frac{||g_n||_T-C_3}{||g_n||_{T'}}\ge \\
\ge h_T(\mu)\liminf_{n\to\infty}\frac{||g_n||_T-C_3}{||g_n||_{T'}}
=h_T(\mu)\liminf_{n\to\infty}\frac{||g_n||_T}{||g_n||_{T'}}\ge \\
\ge h_T(\mu)\inf_{g\in F\setminus\{1\}} \frac{||g||_T}{||g||_{T'}},
\end{gather*}
as required.
\end{proof}

The following statement is an immediate corollary of the explicit formula for the Patterson-Sullivan current in the Outer Space context obtained by
Kapovich and Nagnibeda (\cite{KN}, see Proposition 5.3):
\begin{prop}\label{prop:KN}
Let $T\in cv(F)$ and let $h=h(T)$ be the critical exponent of $T$. Let $\mu_T\in Curr(F)$ be a Patterson-Sullivan current corresponding to $T$.
Then there exist constants $C_1>C_2>0$ such that for any distinct vertices $x$ and $y$ of $T$ we have
\[
 C_2 \exp(-h\, d_T(x,y))\le \mu_T(Cyl_{[x,y]} )\le C_1 \exp(-h\, d_T(x,y)).
\]

\end{prop}

Together with the definitions of geometric entropy and of tameness, Proposition~\ref{prop:KN} immediately implies:

\begin{cor}\label{cor:ps}
Let $T\in cv(F)$ and let $\mu_T\in Curr(F)$ be a Patterson-Sullivan current corresponding to $T$. Let $h=h(T)$ be the critical exponent of $T$.

Then $\mu_T$ is tame and $h_T(\mu_T)=h(T)$.
\end{cor}

\begin{prop}\label{prop:comp}
Let $T\in cv(F)$ and let $h=h(T)$ be the critical exponent of $T$. Let
$\mu_T\in Curr(F)$ be a Patterson-Sullivan current corresponding to
$T$.

Let $T'\in cv(F)$. Then for any $g\in F\setminus\{1\}$ we have
\[
h_{T'}(\mu_T)\le h\frac{||g||_T}{||g||_{T'}}
\]
and therefore
\[
h_{T'}(\mu_T)\le h\inf_{f\in F\setminus\{1\}}\frac{||f||_T}{||f||_{T'}}.
\]
\end{prop}

\begin{proof}
By Proposition~\ref{prop:KN} there exists $C>0$ such that for any distinct vertices $x,y\in T$ we have
\[
\mu(Cyl_{[x,y]})\ge C \exp( -h d_T(x,y)).
\]

Note that for any $x\in T, x'\in T'$ we have
\begin{gather*}
||g||_T=\lim_{n\to\infty} \frac{d_T(x,g^nx)}{n}, \quad ||g||_{T'}=\lim_{n\to\infty} \frac{d_{T'}(x',g^nx')}{n}
\end{gather*}
and hence
\[
\lim_{n\to\infty} \frac{d_T(x,g^nx)}{d_{T'}(x',g^nx')}=\frac{||g||_T}{||g||_{T'}}.
\]

Let $x\in T$ and $x'\in T'$ be arbitrary vertices. Let $\phi: T\to T'$ be an $F$-equivariant quasi-isometry  such that $\phi(x)=x'$. Thus $\phi(g^nx)=g^n x'$ for every $n\in \mathbb Z$. Let $M\ge 1$ be provided by Lemma~\ref{lem:qi}.
For $n\to\infty$ let $[y_n,z_n]\subseteq [x',g^n x']$ be such that $d(x',y_n)=d(z_n, g^nx')=M$. Then by Lemma~\ref{lem:qi} we have
\[
\phi(Cyl_{[x,g^nx]})\subseteq Cyl_{[y_n,z_n]}.
\]
Hence
\[
\mu(Cyl_{[y_n,z_n]})\ge \mu(Cyl_{[x,g^nx]})\ge C \exp(-h d_T(x,g^nx))
\]
and so
\[
\frac{-\log \big( \mu(Cyl_{[y_n,z_n]}) \big)}{d_{T'}(y_n,z_n)}\le \frac{h d_T(x,g^nx) -\log C}{d_{T'}(y_n,z_n)}\le
\frac{h d_T(x,g^nx) -\log C}{d_{T'}(x',g^nx')-2M}.
\]
Hence
\[
h_{T'}(\mu)\le \liminf_{n\to\infty} \frac{h d_T(x,g^nx) -\log C}{d_{T'}(x',g^nx')-2M}=h\frac{||g||_T}{||g||_{T'}},
\]
as required.
\end{proof}

Since Patterson-Sullivan currents are tame,
Proposition~\ref{prop:comp} and Theorem~\ref{thm:lowerbound} imply

\begin{thm}\label{thm:comp1}
Let $T\in cv(F)$ and let $h=h(T)$ be the critical exponent of $T$. Let $\mu_T\in Curr(F)$ be a Patterson-Sullivan current for $T$.
Let $T'\in cv(F)$. Then
\[
h_{T'}(\mu_T)= h(T)\inf_{f\in F\setminus\{1\}}\frac{||f||_T}{||f||_{T'}}.
\]
\end{thm}

\begin{cor}\label{cor:tw}
Let $T\in CV(F)$ and let $h=h(T)$ be the critical exponent of $T$. Let $\mu_T\in Curr(F)$ be a Patterson-Sullivan current for $T$.
Let $T'\in CV(F)$ be such that $T'\ne T$. Then
\[
h_{T'}(\mu_T)<h_T(\mu_T)=h(T).
\]
Thus
\[
h(T)=h_T(\mu_T)=\max_{T'\in CV(T)} h_{T'}(\mu_T)
\]
and this maximum is strict.
\end{cor}
\begin{proof}
A result of Tad White~\cite{Wh} implies that, when $T'\ne T$, there exists a nontrivial $g\in F$ such that $||g||_T<||g||_{T'}$.
Therefore by Theorem~\ref{thm:comp1}
\[
h_{T'}(\mu_T)= h(T)\frac{||g||_T}{||g||_{T'}}<h(T) .
\]
\end{proof}

\begin{cor}\label{cor:inf}
Let $T\in CV(F)$ and let $\mu_T\in Curr(F)$ be a Patterson-Sullivan current for $T$.
Then
\[
\inf_{T'\in CV(F)} h_{T'}(\mu_T)=0.
\]
\end{cor}

\begin{proof}
Recall that $F$ is a free group of rank $k\ge 2$.
Let $A$ be a free basis of $F$ and let $T_A$ be the Cayley graph of $F$ with respect to $A$, where every edge has length $1/k$. Thus $T_A\in CV(F)$.
Let $a\in A$.
There exists a sequence $\phi_n\in Out(F)$ such that $\lim_{n\to\infty} ||\phi_n a||_A=\infty$ and hence
$\lim_{n\to\infty} ||\phi_n a||_{T_A}=\frac{1}{k}\lim_{n\to\infty} ||\phi_n a||_A=\infty$. Put $T_n=\phi_n^{-1} T_A$. Thus $T_n\in CV(F)$ and
\[
||a||_{T_n}=||a||_{\phi_n^{-1} T_A}=||\phi_n a||_{T_A}\longrightarrow_{n\to\infty}\infty.
\]
Therefore by Theorem~\ref{thm:comp1} we have
\[
h_{T_n}(\mu)\le h(T)\frac{||a||_T}{||a||_{T_n}}\longrightarrow_{n\to\infty} 0.
\]
Hence
\[
\inf_{T'\in CV(F)} h_{T'}(\mu)=0,
\]
as required.

\end{proof}

\section{The maximal geometric entropy problem for a fixed tree}\label{sect:hd}

Recall that, as observed in Introduction, the function
$h_T(\cdot): Curr(F)\to \mathbb R, \mu\mapsto h_T(\mu)$ is not
continuous. Nevertheless, it turns out that it is possible to
find the maximal value of $h_T(\cdot)$ on $Curr(F)-\{0\}$.

Recall that if $(X,d)$ is a metric space and $\nu$ is a positive measure on $X$, then the \emph{Hausdorff dimension ${\mathbf {HD}} _X(\nu)$ of $\nu$ with
respect to $X$} is defined as
\[
{\mathbf {HD}} _X(\nu):=\inf\{ {\mathbf {HD}} (S): S\subseteq X\text{ such that } \nu(X-S)=0\}.
\]
Thus ${\mathbf {HD}} _X(\nu)$ is the smallest Hausdorff dimension of a subset of $X$ of full $\nu$-measure.
Note that this obviously implies that ${\mathbf {HD}} _X(\nu)\le {\mathbf {HD}} (X)$.

Let $T\in cv(F)$.
Recall that if $x\in T$, $\xi,\zeta\in \partial T$, we denote by $(\xi|\zeta)_x$ the distance $d_T(x,y)$ where $[x,\xi]\cap
[x,\zeta]=[x,y]$.
Let $x\in T$ be a base-point. The boundary $\partial T$ is metrized by setting $d_x(\xi,\zeta)=\exp(- (\xi|\zeta)_x )$
for $\xi,\zeta\in \partial T$.
It is well-known (see \cite{Coor}) that ${\mathbf {HD}} (\partial T, d_x)=h(T)$.

Recall that, as explained in the Introduction, given a current $\mu\in
Curr(F)$ and $T\in cv(F)$, we introduce a family of measures $\{\mu_x\}_{x\in T}$ on $\partial F$
defined by their values on all the one-sided cylinder subsets of
$\partial F$:
\[
Cyl_{[x,y]}^x := \{\xi\in \partial F:  \text{ the geodesic ray }
  [x,\partial_T(\xi)] \text{ in $T$ begins with } [x,y]\}\subseteq \partial F,
\]
\[\mu_x (Cyl_{[x,y]}^x) := \mu (Cyl_{[x,y]}) .\]
It is not hard to see that if $\mu\ne 0$ then there exists $x\in T$
such that $\mu_x\ne 0$.

\begin{thm}\label{thm:vol}
Let $\mu\in Curr(F)$, $\mu\ne 0$, let $T\in cv(F)$, and let $x\in T$ be such that $\mu_x\ne 0$.
Then
\[
h_T(\mu)\le {\mathbf {HD}} _{\partial T}(\mu_x)\le h(T).
\]
\end{thm}
\begin{proof}
As observed by Kaimanovich in \cite{Kaim98} (see formula (1.3.3)), the following formula holds for the Hausdorff dimension of a measure $\nu$ on $\partial T$ endowed with the metric $d_x$ as above. 
\[
{\mathbf {HD}} _{\partial T}(\nu)={\rm ess\ sup}_{\xi\in \partial T} \liminf_{k\to\infty} \frac{-\log \nu
\left(B_x(\xi,k)\right)}{k}\tag{$\ddag$}
\]

Here $B_x(\xi,k)$ is the set of all $\zeta\in \partial T$ such that
$(\xi|\zeta)_x\ge k$, that is $B_x(\xi,k)=Cyl_{[x,y_k]}^x$
where $[x,y_k]$ is the initial segment of $[x,\xi]$ of length $k$.
The essential supremum in $(\ddag)$ is taken with respect to $\nu$.

Applied to $\mu_x$, formula $(\ddag)$ yields:

\begin{gather*}
h_T(\mu)=\liminf_{d_T(y,z)\to\infty}\frac{-\log \mu (Cyl_{[y,z]})}{d_T(y,z)}\le \\
\le \liminf_{\overset{z\in T}{d_T(x,z)\to\infty}}\frac{-\log \mu (Cyl_{[x,z]})}{d_T(x,z)}\le \\
\le {\rm ess\ sup}_{\xi\in \partial T} \liminf_{\overset{z\in T}{d_T(x,z)\to\infty}} \frac{-\log \mu_x(Cyl_{[x,z]}^x)}{d_T(x,z)}=\\
={\mathbf {HD}} _{\partial T} (\mu_x)\le {\mathbf {HD}} (\partial T)=h(T).
\end{gather*}

\end{proof}

\begin{prop}\label{prop:proj}
Let $T,T'\in cv(F)$ be such that $h:=h(T)=h(T')$. Let $\mu_T$ be a Patterson-Sullivan current corresponding to $T$ and suppose that
$h_{T'}(\mu_T)=h$.
Then $T$ and $T'$ are in the same projective class.
\end{prop}
\begin{proof}

Let $\mu_{T'}\in Curr(F)$ be a Patterson-Sullivan current corresponding to $T'$.
We will first show that $\mu_T$ is absolutely continuous with respect to $\mu_{T'}$.
Let  $\phi:T\to T'$ be an $F$-equivariant $(\lambda,\lambda)$-quasi-isometry, where $\lambda\ge 1$.
By Proposition~\ref{prop:KN} there is a constant $C\ge 1$ such that for any $x,y\in T$ with
$d_T(x,y)\ge 1$
\[
\frac{1}{C} \exp( -h d_T(x,y))\le \mu_T Cyl_{[x,y]} \le C \exp( -h d_T(x,y)).
\]

Let $x',y'\in \phi(T')$ be arbitrary such that $d_{T'}(x',y')\ge \lambda^2+\lambda$.
Let $x,y\in T$ be such that $x'=\phi(x)$, $y'=\phi(y)$.
Since $\mu_T$ is tame, Lemma~\ref{lem:qi} and Proposition~\ref{prop:qi1} imply that there is some
constant $C_1\ge 1$ such that
\[
\frac{1}{C_1} \mu_T (Cyl_{[x',y']}) \le \mu_T (Cyl_{[x,y]})\le C_1 \mu_T (Cyl_{[x',y']}).
\]
Thus
\[
\mu_T (Cyl_{[x',y']})\ge \frac{1}{C_1}\mu_T (Cyl_{[x,y]})\ge \frac{1}{C_1C} \exp( -h d_T(x,y)).
\]

On the other hand, since by assumption $h_{T'}(\mu_T)=h$, it follows
that for any $\epsilon>0$ there exists $C_2=C_2(\epsilon)\ge 1$ such that
\[
\mu_T (Cyl_{[x',y']})\le C_2 \exp( -(h-\epsilon) d_{T'}(x',y')).
\]

Thus
\[
\frac{1}{C_1C} \exp( -h d_T(x,y)) \le C_2 \exp( -(h-\epsilon) d_{T'}(x',y')).
\]
Hence
\[
h d_T(x,y)\ge (h-\epsilon) d_{T'}(x',y')-\log(C_2C_1C).
\]
and so
\[
\frac{h}{h-\epsilon} \ge \limsup_{d_T(x,y)\to\infty} \frac{d_{T'}(x',y')}{d_T(x,y)}.
\]
Since $\epsilon>0$ was arbitrary, it follows that
\[
\limsup_{d_T(x,y)\to\infty} \frac{d_{T'}(x',y')}{d_T(x,y)}\le 1.
\]
Therefore
\[
\sup_{g\in F\setminus\{1\}} \frac{||g||_{T'}}{||g||_T}\le 1.
\]
By Lemma~\ref{lem:stretch} this implies that there is a constant $C_3\ge 1$ such that
\[
d_T(x',y')\le d_{T}(x,y)+C_3.
\]
Then
\begin{gather*}
\mu_T(Cyl_{[x',y']})\le \frac{1}{C_1C} \exp( -h d_T(x,y))\le \\
\frac{1}{C_1C} \exp( -h (d_{T'}(x',y')
-C_3))=\frac{\exp(hC_3)}{C_1C}\exp(-h d_{T'}(x',y'))\le\\
\frac{C'\exp(hC_3)}{C_1C}\mu_{T'}(Cyl_{[x',y']}).
\end{gather*}
The above inequality holds for any $x',y'\in \phi(T)$ with $d_{T'}(x',y')\ge \lambda^2+\lambda$.
 Since $\phi$ is a quasi-isometry and $\mu_T$ is tame, it follows that there exists a constant $C'\ge 1$
 such that for any $x',y'\in T'$ with $d_{T'}(x',y')\ge 1$ we have
\[
\mu_T(Cyl_{[x',y']})\le C'\mu_{T'}(Cyl_{[x',y']}).
\]
Hence $\mu_T$ is absolutely continuous with respect to $\mu_{T'}$.
A result of Furman~\cite{Fur} now implies that the translation length functions $||.||_T$ and $||.||_{T'}$ are scalar multiples of each other, as required.
\end{proof}

\begin{cor}\label{cor:sharp}
Let $T_1,T_2\in cv(F)$ be such two elements that do not lie in the same projective class. Let $\mu_{T_2}$ be a Patterson-Sullivan current for $T_2$.
Then $h_{T_1}(\mu_{T_2})< h(T_1)$.
\end{cor}
\begin{proof}
 After replacing $T_2$  by a scalar multiple of $T_2$ we may assume that $h(T_1)=h(T_2)$. Note that the projective Patterson-Sullivan current depends only on the projective class of an element of $cv(F)$, so that this replacement does not change $\mu_{T_2}$.
Now the statement of the corollary follows immediately from  Theorem~\ref{thm:vol} and Proposition~\ref{prop:proj}.
\end{proof}

\begin{cor}
Let $T,T'\in cv(F)$. Then
\begin{enumerate}
\item
\[
\inf_{g\in F\setminus\{1\}} \frac{||g||_T}{||g||_{T'}}\le \frac{h(T')}{h(T)}\le \sup_{g\in F\setminus\{1\}} \frac{||g||_T}{||g||_{T'}}.
\]
\item Suppose that $T$ and $T'$ are not in the same projective class. Then
\[
\inf_{g\in F\setminus\{1\}} \frac{||g||_T}{||g||_{T'}}< \frac{h(T')}{h(T)}< \sup_{g\in F\setminus\{1\}} \frac{||g||_T}{||g||_{T'}}.
\]
\end{enumerate}
\end{cor}
\begin{proof} $ $

\noindent (1) Let $\mu_T$ be a Patterson-Sullivan current corresponding to $T$.
By Theorem~\ref{thm:comp1} and Theorem~\ref{thm:vol} we have
\[
h(T)\inf_{g\in F\setminus\{1\}} \frac{||g||_T}{||g||_{T'}}=h_{T'}(\mu_T)\le h(T')
\]
and hence \[\inf_{g\in F\setminus\{1\}} \frac{||g||_T}{||g||_{T'}}\le \frac{h(T')}{h(T)}.\]

A symmetric argument shows that \[\inf_{g\in F\setminus\{1\}} \frac{||g||_{T'}}{||g||_{T}}\le \frac{h(T)}{h(T')}.\]
Clearly,
\[
\inf_{g\in F\setminus\{1\}} \frac{||g||_{T'}}{||g||_{T}}=\frac{1}{\displaystyle\sup_{g\in F\setminus\{1\}} \frac{||g||_T}{||g||_{T'}}}
\]
and hence
\[
\sup_{g\in F\setminus\{1\}} \frac{||g||_T}{||g||_{T'}}\ge \frac{h(T')}{h(T)},\]
as required.

The proof of part (2) is exactly the same, but Corollary~\ref{cor:sharp} implies that all the inequalities involved are now strict.
\end{proof}

\section{Relation to measure-theoretic entropy}

In this section we relate the geometric entropy of a current $\mu\in Curr(F)$ with respect to $T=\widetilde \Gamma$, where $\Gamma$ is a finite graph with simplicial metric (every edge has length one), to the measure-theoretic entropy of the corresponding shift-invariant measure $\widehat\mu$ on the (appropriately defined) geodesic flow space of the graph $\Gamma$.

\begin{defn}\label{defn:Omega}
Let $\alpha:F\to \pi_1(\Gamma, p)$ be a simplicial chart. Let $T=\widetilde \Gamma$. We endow both $\Gamma$ and $T$ with simplicial metrics by giving each edge length one. Thus $T\in cv(F)$.
Define $\Omega(\Gamma)$ to be the set of all semi-infinite reduced edge-paths
\[
\gamma=e_1,e_2,\dots, e_n,\dots
\]
in $\Gamma$. Note that $\Omega(\Gamma)$ is naturally identified with the disjoint union of $\#V\Gamma$ copies of $\partial T$, corresponding to the $\#V\Gamma$ different possibilities of the initial vertex of $\gamma\in \Omega(\Gamma)$. We topologize $\Omega(\Gamma)$ accordingly. Thus, topologically, $\Omega(\Gamma)$ is a disjoint union of $\#V\Gamma$ copies of the Cantor set.

There is a natural shift transformation $\sigma:\Omega(\Gamma)\to\Omega(\Gamma)$ defined as
\[
\sigma(e_1,e_2,e_3,\dots)=e_2,e_3,\dots
\]
for every $\gamma=e_1,e_2,e_3,\dots \in \Omega(\Gamma)$. Then $\sigma$ is a continuous transformation and the pair $(\Omega(\Gamma),\sigma)$ is easily seen to be an irreducible subshift of finite type, where the alphabet is the set of oriented edges of $\Gamma$.
\end{defn}

The space $\Omega(\Gamma)$ has a natural set of cylinder sets which generate its Borel sigma-algebra: for every nontrivial reduced edge-path $v$ in $\Gamma$ let
$Cyl_v$ be the set of all $\gamma\in \Omega(\Gamma)$ that have $v$ as an initial segment. It is easy to see that there is a natural affine isomorphism between the space of geodesic currents $Curr(F)$ and the space $\mathcal M(\Gamma)$ of all positive $\sigma$-invariant Borel measures on $\Omega(\Gamma)$ of finite total mass. For a current $\mu\in Curr(F)$ the corresponding measure $\widehat\mu\in \mathcal M(\Gamma)$ is defined by the following condition. For every nontrivial reduced edge-path $v$ in $\Gamma$, let $[x,y]\subseteq T$ be any lift of $v$ to $T$. Then
\[
\widehat\mu(Cyl_v)=\mu(Cyl_{[x,y]}).
\]

Suppose now that $\mu\in Curr(F)$ is normalized so that the corresponding measure $\widehat\mu\in \Omega(\Gamma)$ is a probability measure (recall that multiplication by a positive scalar does not change the geometric entropy). For the probability measure $\widehat\mu$ one can consider its classical \emph{measure-theoretic} entropy $\hbar(\widehat \mu)$ with respect to the shift $\sigma$, also known as the \emph{Kolmogorov-Sinai entropy} or \emph{metric entropy}, see e.g. \cite{Kitchens}.

\begin{thm}\label{thm:ks}
Let $\alpha:F\to \pi_1(\Gamma, p)$ be a simplicial chart. Let $T=\widetilde \Gamma$ and endow both $\Gamma$ and $T$ with simplicial metric.
Let $(\Omega(\Gamma),\sigma)$ be as in Definition~\ref{defn:Omega}.
Let $\mu\in Curr(F)$ be such that the corresponding measure $\widehat\mu\in \mathcal M(\Gamma)$ is a probability measure.
Then
\[
h_{T}(\mu)\le \hbar(\widehat\mu)\le h_{topol}(\sigma)=h(T).
\]
\end{thm}

\begin{proof}
The fact that $h_{topol}(\sigma)=h(T)$ is a straightforward exercise which easily follows from the definitions of both quantities.
The fact that $\hbar(\widehat\mu)\le h_{topol}(\sigma)$ is also completely general for subshifts of finite type (see e.g. Ch 6. of \cite{Kitchens}).
If $h_T(\mu)=0$ then the inequality $h_T(\mu)\le \hbar(\widehat\mu)$ is obvious. Suppose now that $h_T(\mu)>0$.

Note that
\[
\Omega(\Gamma)=\sqcup_{e\in E\Gamma} Cyl_e
\]
is a generating partition with respect to $\sigma$. Therefore the measure-theoretic entropy of $\widehat\mu$ can be computed using this partition and is easily seen to be
\[
\hbar(\widehat\mu)=\lim_{n\to\infty} \sum_{|v|=n}-\frac{\widehat\mu(Cyl_v)\log\widehat\mu(Cyl_v)}{n}.
\]
Suppose now that $0<s<h_T(\mu)$. Then there exists $n_0\ge 1$ such that for any geodesic edge-path $[x,y]$ in $T$ of length $n\ge n_0$ we have
\[
\mu(Cyl_{[x,y]})\le \exp(-sn).
\]
Hence for any reduced edge-path $v$ in $\Gamma$ of length $n\ge n_0$ we have
$\widehat\mu(Cyl_v)\le \exp(-sn)$ and $-\log\widehat\mu(Cyl_v)\ge sn$.
Therefore from the above formula for $\hbar(\widehat\mu)$ we get
\[
\hbar(\widehat\mu)\ge \lim_{n\to\infty}\sum_{|v|=n}\frac{\widehat\mu(Cyl_v) sn}{n}=s,
\]
since for every $n\ge 1$ we have $\Omega(\Gamma)=\sqcup_{|v|=n} Cyl_v$ and so $\sum_{|v|=n}\widehat\mu(Cyl_v)=1$.
Thus we see that for every $0<s<h_T(\mu)$ we have $\hbar(\widehat\mu)\ge s$. Therefore $\hbar(\widehat\mu)\ge h_T(\mu)$ as claimed.

\end{proof}

It is also well-known (again see, for example, Ch. 6 of \cite{Kitchens}) that for irreducible subshifts of finite type
such as $(\Omega(\Gamma),\sigma)$ in Theorem~\ref{thm:ks}, there exists a unique probability
measure $\mu_\Omega\in \mathcal M(\Gamma)$ such that $\hbar(\mu_\Gamma)=h_{topol}(\sigma)=h(T)$.
The measure $\mu_\Omega$ is known as the \emph{measure of maximal entropy}.
We already know that for the Patterson-Sullivan current $\mu_T$ we have $h_T(\mu_T)=h(T)$. Thus,
if $\mu_T$ is normalized so that the corresponding measure $\widehat\mu_T\in \mathcal M(\Gamma)$ is
a probability measure, then by Theorem~\ref{thm:ks} we have $h_T(\mu_T)=\hbar(\widehat\mu_T)=h(T)$.
Hence $\widehat\mu_T$ is the unique measure of maximal geometric entropy $\mu_\Omega$.

\begin{cor}\label{cor:ks} Let $\alpha$, $\Gamma$ and $T$ be as in Theorem~\ref{thm:ks}. Let $\mu_T\in Curr(F)$ be the Patterson-Sullivan current for $T$ normalized so that the corresponding measure $\widehat\mu_T\in \mathcal M(\Gamma)$ is a probability measure.
Then 
\begin{enumerate}
\item $\hbar(\widehat\mu_T)=h(T)$, so that $\widehat\mu_T=\mu_\Omega$, the unique probability measure of maximal entropy;
\item for $\mu\in Curr(F)$ we have $h_T(\mu)=h(T)$ if and only if $\mu$ is proportional to $\mu_T$.
\end{enumerate}
\end{cor}

Theorem~\ref{thm:ks} and Corollary~\ref{cor:ks} yield Theorem~\ref{E} from the Introduction.

\end{document}